\documentclass[a4paper,10pt]{article}
\usepackage{amssymb,amsmath,amsthm}
\usepackage[vcentermath]{youngtab}
\usepackage{verbatim}
\usepackage{graphicx}

\newtheorem{theorem}{Theorem}[section] 
\newtheorem{lemma}[theorem]{Lemma}     
\newtheorem{corollary}[theorem]{Corollary}

\newtheorem{conjecture}{Conjecture}  
\newtheorem*{definition}{definition}

\title{A Proof of the Cameron-Ku Conjecture} 

\author{David Ellis}

\begin{document}
\maketitle

\begin{abstract}
A family of permutations \(\mathcal{A} \subset S_{n}\) is said to be \textit{intersecting} if any two permutations in \(\mathcal{A}\) agree at some point, i.e. for any \(\sigma, \pi \in \mathcal{A}\), there is some \(i\) such that \(\sigma(i)=\pi(i)\). Deza and Frankl \cite{dezafrankl} showed that if \(\mathcal{A} \subset S_n\) is intersecting, then \(|\mathcal{A}| \leq (n-1)!\). Cameron and Ku \cite{cameron} showed that if equality holds, then \(\mathcal{A} = \{\sigma \in S_{n}: \sigma(i)=j\}\) for some \(i\) and \(j\). They conjectured a `stability' version of this result, namely that there exists a constant \(c < 1\) such that if \(\mathcal{A} \subset S_{n}\) is an intersecting family of size at least \(c(n-1)!\), then there exist \(i\) and \(j\) such that every permutation in \(\mathcal{A}\) maps \(i\) to \(j\) (we call such a family `centred'). They also made a stronger `Hilton-Milner' type conjecture, namely, that for \(n \geq 6\), if \(\mathcal{A} \subset S_{n}\) is a non-centred intersecting family, then \(\mathcal{A}\) cannot be larger than the family \(\mathcal{C}=\{\sigma \in S_{n}: \sigma(1)=1, \sigma(i)=i \textrm{ for some } i > 2\} \cup \{(12)\}\), which has size \((1-1/e+o(1))(n-1)!\).

We prove the stability conjecture, and also the Hilton-Milner type conjecture for \(n\) sufficiently large. Our proof makes use of the classical representation theory of \(S_{n}\). One of our key tools will be an extremal result on cross-intersecting families of permutations, namely that for \(n \geq 4\), if \(\mathcal{A},\mathcal{B}\subset S_{n}\) are cross-intersecting, then \(|\mathcal{A}||\mathcal{B}| \leq ((n-1)!)^{2}\). This was a conjecture of Leader \cite{leader}; it was proved for \(n\) sufficiently large by Friedgut, Pilpel and the author in \cite{jointpaper}.\end{abstract}

\section{Introduction}
We work in the symmetric group \(S_{n}\), the group of all permutations of \(\{1,2,\ldots,n\} = [n]\). A family of permutations \(\mathcal{A} \subset S_{n}\) is said to be \textit{intersecting} if any two permutations in \(\mathcal{A}\) agree at some point, i.e. for any \(\sigma, \pi \in \mathcal{A}\), there is some \(i \in [n]\) such that \(\sigma(i)=\pi(i)\).

It is natural to ask: how large can an intersecting family be? The family of all permutations fixing 1 is an obvious example of a large intersecting family of permutations; it has size \((n-1)!\). More generally, for any \(i,j \in [n]\), the collection of all permutations mapping \(i\) to \(j\) is clearly an intersecting family of the same size; we call these the `\textit{1-cosets}' of \(S_{n}\), since they are the cosets of the point-stabilizers.

Deza and Frankl \cite{dezafrankl} showed that if \(\mathcal{A} \subset S_{n}\) is intersecting, then \(|\mathcal{A}| \leq (n-1)!\); this is known as the Deza-Frankl theorem. They gave a short, direct `partitioning' proof: take any \(n\)-cycle \(\rho\), and let \(H\) be the cyclic group of order \(n\) generated by \(\rho\). For any left coset \(\sigma H\) of \(H\), any two distinct permutations in \(\sigma H\) disagree at every point, and therefore \(\sigma H\) contains at most 1 member of \(\mathcal{A}\). Since the left cosets of \(H\) partition \(S_{n}\), it follows that \(|\mathcal{A}| \leq (n-1)!\).

Deza and Frankl conjectured that equality holds only for the 1-cosets of \(S_{n}\). This turned out to be harder than expected; it was eventually proved independently by Cameron and Ku \cite{cameron} and Larose and Malvenuto \cite{larose}; Wang and Zhang \cite{wang} have recently given a shorter proof.

We say that an intersecting family \(\mathcal{A} \subset S_{n}\) is \emph{centred} if there exist \(i,j \in [n]\) such that every permutation in \(\mathcal{A}\) maps \(i\) to \(j\), i.e. \(\mathcal{A}\) is contained within a 1-coset of \(S_{n}\). Cameron and Ku asked how large a \textit{non-centred} intersecting family can be. Experimentation suggests that the further an intersecting family is from being centred, the smaller it must be. The following are natural examples of large non-centred intersecting families:
\\
\begin{itemize}
\item \(\mathcal{B}=\{\sigma \in S_{n}: \sigma \textrm{ fixes at least two points in }[3]\}\).\\
\\
This has size \(3(n-2)!-2(n-3)!\).\\
It requires the removal of \((n-2)!-(n-3)!\) permutations to make it centred.\\
\item \(\mathcal{C} = \{\sigma: \sigma(1)=1,\ \sigma \textrm{ intersects } (1\ 2)\} \cup \{(1\ 2)\}\).\\
\\
\textit{Claim:} \(|\mathcal{C}| = (1-1/e+o(1))(n-1)!\).\\
\textit{Proof of Claim:} Let \(\mathcal{D}_{n} = \{\sigma \in S_{n}:\ \sigma(i)\neq i\ \forall i \in [n]\}\) denote the set of \textit{derangements} of \([n]\) (permutations in \(S_n\) without fixed points); let \(d_{n} = |\mathcal{D}_{n}|\) denote the number of derangements of \([n]\). By the inclusion-exclusion formula,
\[d_{n} = \sum_{i=0}^{n} (-1)^{i} {n \choose i}(n-i)! = n! \sum_{i=0}^{n}\frac{(-1)^{i}}{i!} = n!(1/e + o(1)).\]
Note that a permutation which fixes 1 intersects \((1\ 2)\) if and only if it has a fixed point greater than \(2\). The number of permutations fixing 1 alone is clearly \(d_{n-1}\); the number of permutations fixing 1 and 2 alone is clearly \(d_{n-2}\), so the number of permutations fixing 1 and some other point \(>2\) is \((n-1)!-d_{n-1}-d_{n-2}\). Hence,
\[|\mathcal{C}| = (n-1)!-d_{n-1}-d_{n-2} = (1-1/e+o(1))(n-1)!\]
as required.\\
\\
Note that \(\mathcal{C}\) can be made centred just by removing \((1\ 2)\).
\end{itemize}
\vspace{\baselineskip}

For \(n \leq 5\), \(\mathcal{B}\) and \(\mathcal{C}\) have the same size; for \(n \geq 6\), \(\mathcal{C}\) is larger. Cameron and Ku \cite{cameron} conjectured that for \(n \geq 6\), \(\mathcal{C}\) has the largest possible size of any non-centred intersecting family. Further, they conjectured that if \(\mathcal{A} \subset S_n\) is a non-centred intersecting family of the same size as \(\mathcal{C}\), then \(\mathcal{A}\) must be a `double translate' of \(\mathcal{C}\), meaning that there exist \(\pi,\tau \in S_{n}\) such that \(\mathcal{A}=\pi \mathcal{C} \tau\). Note that if \(\mathcal{F} \subset S_{n}\), any double translate of \(\mathcal{F}\) has the same size as \(\mathcal{F}\), is intersecting if and only if \(\mathcal{F}\) is, and is centred if and only if \(\mathcal{F}\) is. Double-translation will be our notion of `isomorphism' for intersecting families of permutations. 

We prove the Cameron-Ku conjecture for all sufficiently large \(n\). This implies the weaker `stability' conjecture of Cameron and Ku \cite{cameron} --- namely, that there exists a constant \(c>0\) such that any intersecting family \(\mathcal{A} \subset S_{n}\) of size at least \((1-c)(n-1)!\) is centred. We prove the latter using a slightly shorter argument.

Our proof makes use of the classical representation theory of \(S_{n}\). One of our key tools will be an extremal result on cross-intersecting families of permutations. A pair of families of permutations \(\mathcal{A},\mathcal{B} \subset S_{n}\) is said to be \textit{cross-intersecting} if for any \(\sigma \in \mathcal{A}, \tau \in \mathcal{B}\), \(\sigma\) and \(\tau\) agree at some point, i.e. there is some \(i \in [n]\) such that \(\sigma(i)=\tau(i)\). Leader \cite{leader} conjectured that for \(n \geq 4\), if \(\mathcal{A},\mathcal{B}\) are cross-intersecting, then \(|\mathcal{A}||\mathcal{B}| \leq ((n-1)!)^{2}\), with equality if and only if \(\mathcal{A}=\mathcal{B}=\{\sigma \in S_{n}: \sigma(i)=j\}\) for some \(i,j \in [n]\). Note that the statement does not hold for \(n=3\), as the pair
\[\mathcal{A} = \{(1),(123),(321)\},\ \mathcal{B} = \{(12),(23),(31)\}\]
is cross-intersecting with \(|\mathcal{A}||\mathcal{B}| = 9\).

A \(k\)-cross-intersecting generalization of Leader's conjecture was proved by Friedgut, Pilpel and the author in \cite{jointpaper}, for \(n\) sufficiently large depending on \(k\). In order to prove the Cameron-Ku conjecture for \(n\) sufficiently large, we could in fact make do with the \(k=1\) case of this result. For completeness, however, we sketch a simpler proof of Leader's conjecture for all \(n \geq 4\), based on the eigenvalues of the derangement graph rather than those of the weighted graph constructed in \cite{jointpaper}. Interestingly, no combinatorial proof of Leader's conjecture is known.

There is a close analogy between intersecting families of permutations and intersecting families of \(r\)-sets, which we now describe. As usual, let \([n]^{(r)}\) denote the set of all \(r\)-element subsets (`\(r\)-sets') of \([n]\). We say that a family \(\mathcal{A} \subset [n]^{(r)}\) is \textit{intersecting} if any two of its sets have nonempty intersection. The classical Erd\H{o}s-Ko-Rado theorem states that if \(r < n/2\), then the largest intersecting families of \(r\)-subsets of \([n]\) are the `stars', meaning the families of the form \(\{x \in [n]^{(r)}:\ i \in x\}\) for \(i \in [n]\). This corresponds to the fact that the largest intersecting families of permutations in \(S_n\) are the 1-cosets.

We say that an intersecting family of \(r\)-sets is \textit{trivial} if there is an element in all of its sets. Hilton and Milner \cite{hiltonmilner} proved that for \(r \geq 4\) and \(n > 2r\), if \(\mathcal{A} \subset [n]^{(r)}\) is a non-trivial intersecting family of maximum size, then
\[\mathcal{A}=\{x \in [n]^{(r)}: i \in [n],\ x \cap y \neq \emptyset\} \cup \{y\}\]
for some \(i \in [n]\) and some \(r\)-set \(y\) not containing \(i\), so it can be made into a trivial family by removing just one \(r\)-set. The Cameron-Ku conjecture is an exact analogue of this for permutations.

\section{Cross-intersecting families of permutations}
Our aim in this section is to prove Leader's conjecture: if \(n \geq 4\), and \(\mathcal{A},\mathcal{B} \subset S_n\) are cross-intersecting, then \(|\mathcal{A}||\mathcal{B}| \leq ((n-1)!)^2\). We do this by applying a `cross-independent' analogue of Hoffman's eigenvalue bound to the derangement graph; the eigenvalues of the derangement graph are analysed using the representation theory of \(S_n\).

The {\em derangement graph} is the graph \(\Gamma\) with vertex-set \(S_n\), where two permutations are joined if they disagree everywhere, i.e.
\[V(\Gamma)=S_n,\quad E(\Gamma) = \{\sigma \tau:\ \sigma,\tau \in S_n,\ \sigma(i) \neq \tau(i)\ \forall i \in [n]\}.\]
Recall that if \(G\) is a finite group, and \(S \subset G\) is inverse-closed (\(S^{-1}=S\)), the {\em Cayley graph on \(G\) generated by \(S\)} is the graph with vertex-set \(G\), where we join \(g\) to \(gs\) for each \(g \in G\) and each \(s \in S\). Clearly, the derangement graph \(\Gamma\) is the Cayley graph on \(S_n\) generated by \(\mathcal{D}_n\), the set of derangments of \([n]\), so it is \(d_n\)-regular. Of course, an intersecting family of permutations in \(S_n\) is precisely an independent set in \(\Gamma\), and \(\mathcal{A},\mathcal{B} \subset S_n\) are cross-intersecting if and only if there are no edges of \(\Gamma\) between them.

Hoffman's theorem provides an upper bound on the maximum size of an independent set in a regular graph in terms of the minimum eigenvalue of the adjacency matrix of the graph. Recall that if \(H = (V,E)\) is an \(N\)-vertex graph, the {\em adjacency matrix} of \(H\) is the 0-1 matrix with rows and columns indexed by \(V\), and with
\[A_{x,y} = \left\{\begin{array}{rl} 1 & \textrm{if }xy \in E(H);\\
                     0 & \textrm{otherwise}.
\end{array}\right.\]

Since \(A\) is a real symmetric matrix, all its eigenvalues are real, and for any inner product on \(\mathbb{R}^V\), we can find an orthonormal basis of \(\mathbb{R}^V\) consisting of real eigenvectors of \(A\). Let \(\lambda_1 \geq \lambda_2 \geq \ldots \geq \lambda_N = \lambda_{\min}\) be the eigenvalues of \(A\), repeated with their multiplicities. It is easy to see that if \(H\) is \(d\)-regular, then \(\lambda_1 = d\). Hoffman's bound is as follows:

\begin{theorem}(Hoffman's bound)\\
\label{thm:hoffman}
Let \(H = (V,E)\) be a \(d\)-regular graph, and let \(A\) be the adjacency matrix of \(H\). Let \(\lambda_{\min}\) denote the least eigenvalue of \(A\). If \(X \subset V(H)\) is an independent set in \(H\), then
\[\frac{|X|}{|V|} \leq \frac{-\lambda_{\min}}{d - \lambda_{\min}}.\]
If equality holds, then the characterstic vector \(v_{X}\) of \(X\) satisfies:
\[v_X - \tfrac{|X|}{|V|} \boldsymbol{1} \in \textrm{Ker}(A - \lambda_{\min}I).\]
\end{theorem}

Ku conjectured that the minimum eigenvalue of the derangement graph is \(-d_n/(n-1)\). This was first proved by Renteln \cite{renteln}, using symmetric functions. Substituting this value into Hoffman's bound implies that an intersecting family of permutations in \(S_n\) has size at most \((n-1)!\), recovering the theorem of Deza and Frankl.

To deal with cross-intersecting families, we first prove an analogue of Hoffman's bound for `cross-independent' sets; this is a variant of a result in \cite{kaplan}.

\begin{theorem} \label{thm:crosshoffman}
(i) Let \(H=(V,E)\) be a \(d\)-regular graph on \(N\) vertices, whose adjacency matrix \(A\) has eigenvalues \(\lambda_{1}=d \geq \lambda_{2} \geq \ldots \geq \lambda_{N}\). Let \(\nu = \max(|\lambda_{2}|,|\lambda_{N}|)\). Suppose \(X\) and \(Y\) are sets of vertices of \(\Gamma\) with no edges between them, i.e. \(xy \notin E(\Gamma)\) for every \(x \in X\) and \(y \in Y\). Then
\begin{equation}
\label{eq:bestbound}
\sqrt{|X||Y|} \leq \frac{\nu}{d+\nu}N.
\end{equation}
(ii) Suppose further that \(|\lambda_{2}| \neq |\lambda_{N}|\), and let \(\lambda'\) be the larger in modulus of the two. Let \(v_{X},v_{Y} \in \mathbb{R}^V\) be the characteristic vectors of \(X,Y\) and let \(\mathbf{f}\) denote the all-1's vector in \(\mathbb{R}^{V}\). If equality holds in (\ref{eq:bestbound}), then \(|X|=|Y|\), and we have
\[v_{X}-(|X|/N)\mathbf{f},\ v_{Y}-(|Y|/N)\mathbf{f}\ \in \ \textrm{Ker}(A-\lambda'I).\]
\end{theorem}
\begin{proof}
This is a straightforward extension of the proof of Hoffman's theorem, with an application of the Cauchy-Schwarz inequality. Equip \(\mathbb{R}^{V}\) with the inner product:
\[\langle u,v \rangle = \frac{1}{N}\sum_{i=1}^{N} u(i)v(i),\]
and let
\[||u|| = \sqrt{\frac{1}{N}\sum_{i=1}^{N}u(i)^{2}}\]
be the induced Euclidean norm. Let \(u_{1} = \mathbf{f},u_{2},\ldots,u_{N}\) be an orthonormal basis of real eigenvectors of \(A\) corresponding to the eigenvalues \(\lambda_{1}=d,\lambda_{2},\ldots,\lambda_{N}\). Let \(X,Y\) be as above; let \(\alpha = |X|/N\), and let \(\beta = |Y|/N\). Write
\[v_{X}=\sum_{i=1}^{N} \xi_{i} u_{i},\quad v_{Y} = \sum_{i=1}^{N} \eta_{i} u_{i}\]
as linear combinations of the eigenvectors of \(A\). We have \(\xi_{1}=\alpha\), \(\eta_{1}=\beta\), and
\[\sum_{i=1}^{N} \xi_{i}^{2} = ||v_{X}||^{2} = |X|/N = \alpha,\quad \sum_{i=1}^{N} \eta_{i}^{2} = ||v_{Y}||^{2} = |Y|/N = \beta.\]
Since there are no edges of \(H\) between \(X\) and \(Y\), we have:
\begin{equation}
\label{eq:sum}
0=\sum_{x \in X,y \in Y}A_{x,y}=v_{Y}^{\top} A v_{X} = \sum_{i=1}^{N} \lambda_{i}\xi_{i}\eta_{i} = d \alpha \beta + \sum_{i=2}^{N} \lambda_{i}\xi_{i} \eta_{i} \geq d \alpha \beta - \nu \left|\sum_{i=2}^{N}\xi_{i}\eta_{i}\right|.
\end{equation}
Provided \(|\lambda_{2}| \neq |\lambda_{N}|\), if we have equality above, then \(\xi_{i} = \eta_{i} = 0\) unless \(\lambda_{i}=d\) or \(\lambda'\), so \(v_{X}-\alpha \mathbf{f}\) and \(v_{Y}-\beta \mathbf{f}\) are both \(\lambda'\)-eigenvectors of \(A\).

The Cauchy-Schwarz inequality gives:
\[\left|\sum_{i=2}^{N}\xi_{i}\eta_{i}\right| \leq \sqrt{\sum_{i=2}^{N} \xi_{i}^{2} \sum_{i=2}^{N} \eta_{i}^{2}} = \sqrt{(\alpha-\alpha^{2})(\beta-\beta^{2})}.\]
Substituting this into (\ref{eq:sum}) gives:
\[d \alpha\beta \leq \nu \sqrt{(\alpha-\alpha^{2})(\beta-\beta^{2})},\]
so
\[\frac{\alpha\beta}{(1-\alpha)(1-\beta)} \leq (\nu/d)^{2}.\]
By the AM/GM inequality, 
\((\alpha+\beta)/2 \geq \sqrt{\alpha\beta}\) with equality if and only if \(\alpha = \beta\), so
\[\frac{\alpha\beta}{(1-\sqrt{\alpha\beta})^{2}} = \frac{\alpha\beta}{1-2 \sqrt{\alpha\beta} + \alpha \beta} \leq \frac{\alpha\beta}{1-\alpha-\beta + \alpha\beta} \leq (\nu/d)^{2},\]
implying that
\[\sqrt{\alpha\beta} \leq \frac{\nu}{d+\nu}.\]
Hence, we have
\[\sqrt{|X||Y|} \leq \frac{\nu}{d+\nu}N.\]
Provided \(|\lambda_{2}| \neq |\lambda_{N}|\), we have equality only if \(|X|=|Y|=\frac{\nu}{d+\nu}N\) and \(v_{X}-\alpha \mathbf{f}\) and \(v_{Y}-\beta \mathbf{f}\) are both \(\lambda'\)-eigenvectors of \(A\), completing the proof.
\end{proof}
We will show that for \(n \geq 5\), the derangement graph satisfies the hypotheses of this result with \(\nu = d_{n}/(n-1)\); in fact, \(\lambda_{N} = -\frac{d_{n}}{n-1}\), and \(\max_{i \neq 1,N}|\lambda_i| = O((n-2)!)\). The derangement graph is a {\em normal} Cayley graph, meaning that its generating set is a union of conjugacy-classes; as is well-known, there is a particularly nice correspondence between eigenspaces of normal Cayley graphs, and irreducible representations of the group.

Note that the least eigenvalue of the derangement graph was first calculated by Renteln \cite{renteln}, using symmetric functions, and somewhat later by Ku and Wales \cite{kuwales}, and by Godsil and Meagher \cite{meagher}. We analyse the eigenvalues of the derangement graph differently, employing a convenient trick known as the `trace method' to bound all eigenvalues of high multiplicity. The idea of the trace method is simple: if \(H\) is a graph on \(N\) vertices, whose adjacency matrix \(A\) has eigenvalues \(\lambda_{1} \geq \lambda_{2} \geq \ldots \geq \lambda_{N}\), then for any \(k \in \mathbb{N}\),
\[\textrm{Trace}(A^k) = \sum_{i=1}^{N}\lambda_i^{k}.\]
On the other hand, \(\textrm{Trace}(A^k)\) is also the number of closed walks of length \(k\) in \(H\). We will apply this with \(k=2\); \(\textrm{Trace}(A^2)\) is simply twice the number of edges of \(H\).

\subsubsection*{Background on general representation theory}
We now recall the concepts we need from general representation theory. Readers familiar with representation theory may wish to skip this section; others may wish to refer to \cite{isaacs} for additional information.

Let \(G\) be a finite group, and let \(F\) be a field. A {\em representation of \(G\) over \(F\)} is a pair \((\rho,V)\), where \(V\) is a finite-dimensional vector space over \(F\), and \(\rho:\ G \to GL(V)\) is a group homomorphism from \(G\) to the group of all invertible linear endomorphisms of \(V\). The vector space \(V\), together with the linear action of \(G\) defined by \(gv = \rho(g)(v)\), is sometimes called an \(FG\)-{\em module}.  A {\em homomorphism} between two representations \((\rho,V)\) and \((\rho',V')\) is a linear map \(\phi:V \to V'\) such that \(\phi(\rho(g)(v)) = \rho'(g)(\phi(v))\) for all \(g \in G\) and \(v \in V\). If there exists such a \(\phi\) which is bijective, then the two representations are said to be {\em isomorphic}, or {\em equivalent}, and we write \((\rho,V) \cong (\rho',V')\). If \(\dim(V)=n\), we say that \(\rho\) {\em has dimension} \(n\).

The representation \((\rho,V)\) is said to be {\em irreducible} if it has no proper subrepresentation, i.e. there is no proper subspace of \(V\) which is \(\rho(g)\)-invariant for all \(g \in G\). {\em Schur's lemma} states that if \((\rho,V)\) is an irreducible representation of \(G\) over \(\mathbb{C}\), then the only linear endomorphisms of \(V\) which commute with \(\rho\) are scalar multiples of the identity.\footnote{Recall that a linear endomorphism \(\alpha\) of \(V\) is said to {\em commute} with \(\rho\) if \(\alpha\) commutes with \(\rho(g)\) for all \(g \in G\), i.e. \(\rho(g) \circ \alpha = \alpha \circ \rho(g)\) for all \(g \in G\).}

When \(F = \mathbb{R}\) or \(\mathbb{C}\), it turns out that there are only finitely many isomorphism classes of irreducible representations of \(G\), and {\em any} representation of \(G\) is isomorphic to a direct sum of irreducible representations of \(G\).

If \((\rho,V)\) is a representation of \(V\), the {\em character} \(\chi_{\rho}\) of \(\rho\) is the map
\begin{eqnarray*}
\chi_{\rho}:  G & \to & F;\\
 \chi_{\rho}(g) &=& \textrm{Trace} (\rho(g)).
\end{eqnarray*}
The usefulness of characters lies in the fact that {\em two complex representations are isomorphic if and only if they have the same character.}

Given two representations \((\rho,V)\) and \((\rho',V')\) of \(G\), we can form their direct sum, the representation \((\rho \oplus \rho',V \oplus V')\), and their tensor product, the representation \((\rho \otimes \rho',V \otimes V')\). We have \(\chi_{\rho \oplus \rho'} = \chi_{\rho}+ \chi_{\rho'}\), and \(\chi_{\rho \otimes \rho'} = \chi_{\rho} \cdot \chi_{\rho'}\) (the pointwise product).

The {\em group algebra} \(FG\) denotes the \(F\)-vector space with basis \(G\) and multiplication defined by extending the group multiplication linearly. In other words,
\[FG = \left\{\sum_{g \in G}x_{g}g:\ x_{g} \in F\ \forall g \in G\right\},\]
and
\[\left(\sum_{g \in G} x_{g}g\right)\left(\sum_{h\in G}y_{h}h\right) = \sum_{g,h \in G} x_{g}y_{h} (g h).\]
As a vector space, \(FG\) may be identified with \(F[G]\), the vector-space of all \(F\)-valued functions on \(G\), by identifying \(\sum_{g \in G} x_g g\) with the function \(g \mapsto x_g\). The representation defined by
\[\rho(g)(x) = gx\quad (g \in G,\ x \in FG)\]
is called the {\em left regular representation} of \(G\); the corresponding \(FG\)-module is called the {\em group module}.

Let \(\Gamma\) be a graph on \(G\), and let \(A\) be the adjacency matrix of \(\Gamma\). We may consider \(A\) as a linear operator on either \(\mathbb{R}[G]\) or \(\mathbb{C}[G]\); it makes no difference to the eigenvalues, but the latter is more convenient in general. We have
\[A f(g) = \sum_{h\in G:\atop gh \in E(\Gamma)} f(h)\quad \forall f \in \mathbb{C}[G],\ g \in G.\]

If \(\Gamma\) is a Cayley graph on \(G\) with (inverse-closed) generating set \(S\), then idenfying \(\mathbb{C}G\) with \(\mathbb{C}[G]\), the adjacency matrix of \(\Gamma\) acts on \(\mathbb{C}G\) by right-multiplication by \(\sum_{s \in S} s\):

\begin{eqnarray*}
A \left(\sum_{g \in G} x_g g\right) & = & \sum_{g \in G} \sum_{s \in S} x_{gs} g\\
& = & \sum_{h \in G} \sum_{s \in S} x_h (h s^{-1})\\
& = & \sum_{h \in G} \sum_{s \in S} x_h (h s)\\
& = & \left(\sum_{g \in G} x_g g\right) \left(\sum_{s \in S}s\right).
\end{eqnarray*}

If \(\Gamma\) is a normal Cayley graph, then \(\sum_{s \in S} s\) lies in the {\em centre} of \(\mathbb{C}G\) --- i.e., it commutes with every \(x \in \mathbb{C}G\). This leads, via Schur's lemma, to an explicit 1-1 correspondence between the eigenvalues of \(\Gamma\) and the isomorphism classes of irreducible representations of \(G\): 

\begin{theorem}(Schur; Babai; Diaconis, Shahshahani; Roichman.)\\
\label{thm:normalcayley}
Let \(G\) be a finite group, let \(S \subset G\) be an inverse-closed, conjugation-invariant subset of \(G\), and let \(\Gamma\) be the Cayley graph on \(G\) with generating set \(S\). Let \(A\) denote the adjacency matrix of \(\Gamma\). Let \((\rho_{1},V_{1}),\ldots,(\rho_{k},V_{k})\) be a complete set of non-isomorphic irreducible representations of \(G\) --- i.e., containing one representative from each isomorphism class of irreducible representations of \(G\). Let \(U_{i}\) be the sum of all submodules of the group module \(\mathbb{C}G\) which are isomorphic to \(V_{i}\). We have
\[\mathbb{C}G = \bigoplus_{i=1}^{k}U_{i},\]
and each \(U_{i}\) is an eigenspace of \(A\) with dimension \(\dim(V_{i})^{2}\) and eigenvalue
\[\lambda_{V_{i}} = \frac{1}{\dim(V_{i})}\sum_{s \in S} \chi_{i}(s),\]
where \(\chi_{i}(g) = \textrm{Trace}(\rho_{i}(g))\) denotes the {\em character} of the irreducible representation \((\rho_{i},V_{i})\).
\end{theorem}

Given \(x \in \mathbb{C}G\), its projection onto the eigenspace \(U_{i}\) can be found as follows. Write \(\textrm{Id} = \sum_{i=1}^{k}e_{i}\) where \(e_{i} \in U_{i}\) for each \(i \in [k]\). The \(e_{i}\)'s are called the \textit{primitive central idempotents} of \(\mathbb{C}G\); \(U_{i}\) is the two-sided ideal of \(\mathbb{C}G\) generated by \(e_{i}\), and \(e_{i}\) is given by the following formula:
\begin{equation}
\label{eq:idempotents}
e_{i} = \frac{\dim(V_{i})}{|G|} \sum_{g \in G} \chi_{i}(g^{-1}) g.
\end{equation}
For any \(x \in \mathbb{C}G\), \(x = \sum_{i=1}^{k}e_{i}x\) is the unique decomposition of \(x\) into a sum of elements of the \(U_{i}\)'s; in other words, the projection of \(x\) onto \(U_{i}\) is \(e_{i}x\).

\subsubsection*{Background on the representation theory of the symmetric group}
We now collect the results we need from the representation theory of \(S_{n}\); as in \cite{jointpaper}, our treatment follows \cite{JamesKerber} and \cite{sagan}. Readers who are familiar with the representation theory of \(S_{n}\) may wish to skip this section.

A \emph{partition} of \(n\) is a non-increasing sequence of positive integers summing to \(n\), i.e. a sequence $\alpha = (\alpha_1, \ldots, \alpha_k)$ with \(\alpha_{1} \geq \alpha_{2} \geq \ldots \geq \alpha_{k} \geq 1\) and \(\sum_{i=1}^{k} \alpha_{i}=n\); we write \(\alpha \vdash n\). For example, \((3,2,2) \vdash 7\); we sometimes use the shorthand \((3,2,2) = (3,2^{2})\). 

The \emph{cycle-type} of a permutation \(\sigma \in S_{n}\) is the partition of \(n\) obtained by expressing \(\sigma\) as a product of disjoint cycles and listing its cycle-lengths in non-increasing order. Two permutations in \(S_n\) are conjugate if and only if they have the same cycle-type, so the conjugacy classes of \(S_n\) are in explicit 1-1 correspondence with the partitions of \(n\). Moreover, there is an explicit 1-1 correspondence between partitions of \(n\) and isomorphism classes of irreducible representations of \(S_{n}\), which we now describe.

Let \(\alpha = (\alpha_{1}, \ldots, \alpha_{k})\) be a partiton of \(n\). The \emph{Young diagram} of $\alpha$ is
  an array of $n$ cells, having $k$ left-justified rows where row $i$
  contains $\alpha_i$ cells. For example, the Young diagram of the partition \((3,2^{2})\) is
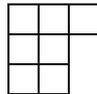
\[\yng(3,2,2)\]
\\

An \(\alpha\)-\emph{tableau} is produced by placing the numbers \(1,2,\ldots,n\) into the cells of the Young diagram of \(\alpha\) in some order; for example,
\[\young(617,54,32)\]
\\
is a \((3,2^{2})\)-tableau. Two \(\alpha\)-tableaux are said to be \emph{row-equivalent} if for each row, they have the same numbers in that row. If an \(\alpha\)-tableau \(t\) has rows \(R_{1},\ldots,R_{k} \subset [n]\) and columns \(C_{1},\ldots,C_{l} \subset [n]\), we let \(R_{t} = S_{R_{1}} \times S_{R_{2}} \times \ldots \times S_{R_{k}}\) be the row-stablizer of \(t\) and \(C_t = S_{C_{1}} \times S_{C_{2}} \times \ldots \times S_{C_{l}}\) be the column-stabilizer.

An \(\alpha\)-\emph{tabloid} is an \(\alpha\)-tableau with unordered row entries (or formally, a row-equivalence class of \(\alpha\)-tableaux); given a tableau \(t\), we write \([t]\) for the tabloid it produces. For example, the \((3,2^{2})\)-tableau above produces the following \((3,2^{2})\)-tabloid\\
\begin{center}
\begin{tabular}{ccc} \{1 & 6 & 7\}\\
\{4 &\ 5\} &\\
\{2 &\ 3\} &\\
\end{tabular}\\
\end{center}
\vspace{\baselineskip}
Consider the natural left action of \(S_{n}\) on the set \(X^{\alpha}\) of all \(\alpha\)-tabloids; let \(M^{\alpha} = \mathbb{C}[X^{\alpha}]\) be the corresponding permutation module, i.e. the complex vector space with basis \(X^{\alpha}\) and \(S_{n}\) action given by extending this action linearly. Given an \(\alpha\)-tableau \(t\), we define the corresponding \(\alpha\)-\textit{polytabloid}
\[e_{t} := \sum_{\pi \in C_{t}} \epsilon(\pi) \pi[t].\]
We define the \textit{Specht module} \(S^{\alpha}\) to be the submodule of \(M^{\alpha}\) spanned by the \(\alpha\)-polytabloids:
\[S^{\alpha} = \textrm{Span}\{e_t:\ t \textrm{ is an } \alpha\textrm{-tableau}\}.\]
A central observation in the representation theory of \(S_{n}\) is that \textit{the Specht modules are a complete set of pairwise non-isomorphic, irreducible representations of} $S_n$. Hence, any irreducible representation \(\rho\) of \(S_{n}\) is isomorphic to some \(S^{\alpha}\). For example, \(S^{(n)} = M^{(n)}\) is the trivial representation; \(M^{(1^{n})}\) is the left-regular representation, and \(S^{(1^{n})}\) is the sign representation \(S\).

We say that a tableau is \textit{standard} if the numbers strictly increase along each row and down each column. It turns out that for any partition \(\alpha\) of \(n\),
\[\{e_{t}: t \textrm{ is a standard }\alpha \textrm{-tableau}\}\]
is a basis for the Specht module \(S^{\alpha}\).

Given a partition \(\alpha\) of \(n\), for each cell \((i,j)\) in its Young diagram, we define the {\em hook length} \((h_{i,j}^{\alpha})\) to be the number of cells in its `hook' (the set of cells in the same row to the right of it or in the same column below it, including itself) --- for example, the hook lengths of \((3,2^{2})\) are as follows:
\[\young(541,32,21)\]
\\

The dimension \(f^{\alpha}\) of the Specht module \(S^{\alpha}\) is given by the following formula
\begin{equation}
\label{eq:hook}
f^{\alpha} = n!/\prod{(\textrm{hook lengths of } \alpha)}.
\end{equation}

From now on we will write \([\alpha]\) for the equivalence class of the irreducible representation \(S^{\alpha}\), \(\chi_{\alpha}\) for the irreducible character \(\chi_{S^{\alpha}}\), and \(\xi_{\alpha}\) for the character of the permutation representation \(M^{\alpha}\). Notice that the set of \(\alpha\)-tabloids form a basis for \(M^{\alpha}\), and therefore \(\xi_{\alpha}(\sigma)\), the trace of the corresponding permutation representation at \(\sigma\), is precisely the number of \(\alpha\)-tabloids fixed by \(\sigma\).

If \(U \in [\alpha],\ V \in [\beta]\), we define \([\alpha]+[\beta]\) to be the equivalence class of \(U\oplus V\), and \([\alpha] \otimes [\beta]\) to be the equivalence class of \(U \otimes V\).

The Branching Theorem (see \cite{JamesKerber} \S 2.4) states that for any partition \(\alpha\) of \(n\), the restriction \([\alpha] \downarrow S_{n-1}\) is isomorphic to a direct sum of those irreducible representations \([\beta]\) of \(S_{n-1}\) such that the Young diagram of \(\beta\) can be obtained from that of \(\alpha\) by deleting a single cell, i.e., if \(\alpha^{i-}\) is the partition whose Young diagram is obtained by deleting the cell at the end of the \(i\)th row of that of \(\alpha\), then
\begin{equation}
\label{eq:restriction}
[\alpha] \downarrow S_{n-1} = \sum_{i: \alpha_{i} > \alpha_{i-1}}[\alpha^{i-}].
\end{equation}
For example, if \(\alpha = (3,2^{2})\), we obtain:
\[[3,2^{2}] \downarrow S_{6} = \left[\ \yng(2,2,2)\ \right]+\left[\ \yng(3,2,1)\ \right] = [2^{3}]+[3,2,1].\]

For any partition \(\alpha\) of \(n\), we have \(S^{(1^{n})} \otimes S^{\alpha} \cong S^{\alpha'}\), where \(\alpha'\) is the \textrm{transpose} of \(\alpha\), the partition of \(n\) with Young diagram obtained by interchanging rows with columns in the Young diagram of \(\alpha\). Hence, \([1^{n}] \otimes [\alpha] = [\alpha']\), and \(\chi_{\alpha'} = \epsilon \cdot \chi_{\alpha}\). For example, we obtain:
\[[3,2,2]\otimes [1^{7}] = \left[\ \yng(3,2,2)\ \right]' = \left[\ \yng(3,3,1)\ \right] = [3,3,1].\]

We now explain how the permutation modules \(M^{\beta}\) decompose into irreducibles.

\begin{definition}
Let \(\alpha,\beta\) be partitions of \(n\). A \emph{generalized} \(\alpha\)-\emph{tableau} is produced by replacing each dot in the Young diagram of \(\alpha\) with a number between 1 and \(n\); if a generalized \(\alpha\)-tableau has \(\beta_{i}\) \(i\)'s \((1 \leq i \leq n)\) it is said to have \emph{content} \(\beta\). A generalized \(\alpha\)-tableau is said to be \emph{semistandard} if the numbers are non-decreasing along each row and strictly increasing down each column.
\end{definition}

  \begin{definition} \label{def:kostka}
  Let $\alpha, \beta$ be partitions of $n$. The {\em Kostka number}
  $K_{\alpha,\beta}$ is the number of semistandard generalized $\alpha$-tableaux
  with content $\beta$.
\end{definition}

\textit{Young's Rule} states that for any partition \(\beta\) of \(n\), the permutation module \(M^{\beta}\) decomposes into irreducibles as follows:
\[M^{\beta} \cong \oplus_{\alpha \vdash n} K_{\alpha, \beta} S^{\alpha}.\]

For example, \(M^{(n-1,1)}\), which corresponds to the natural permutation action of \(S_{n}\) on \([n]\), decomposes as
\[M^{(n-1,1)} \cong S^{(n-1,1)} \oplus S^{(n)},\]
and therefore
\begin{equation}
\label{eq:permcharacter}
\xi_{(n-1,1)} = \chi_{(n-1,1)}+1.
\end{equation}

We now return to considering the derangement graph. Write \(U_{\alpha}\) for the sum of all copies of \(S^{\alpha}\) in \(\mathbb{C}S_{n}\). Note that \(U_{(n)} = \textrm{Span}\{\boldsymbol{f}\}\) is the subspace of constant vectors in \(\mathbb{C}S_{n}\). Applying Theorem \ref{thm:normalcayley} to the derangement graph \(\Gamma\), we have
\[\mathbb{C}S_{n}=\bigoplus_{\alpha \vdash n} U_{\alpha},\]
and each \(U_{\alpha}\) is an eigenspace of the derangement graph, with dimension \(\dim(U_{\alpha}) = (f^{\alpha})^{2}\) and corresponding eigenvalue
\begin{equation}
\label{eq:evalformula}
\lambda_{\alpha} = \frac{1}{f^{\alpha}} \sum_{\sigma \in \mathcal{D}_{n}} \chi_{\alpha}(\sigma).
\end{equation}
We will use the following result, a variant of a result in \cite{JamesKerber}; for the reader's convenience, we include a proof using the Branching Theorem and the Hook Formula.
\begin{lemma}
\label{lemma:lowdimreps}
For \(n \geq 9\), the only Specht modules \(S^{\alpha}\) of dimension \(f^{\alpha} <{n-1 \choose 2}-1\) are as follows:
\begin{itemize}
\item \(S^{(n)}\) (the trivial representation), dimension 1;
\item \(S^{(1^{n})}\) (the sign representation \(S\)), dimension 1;
\item \(S^{(n-1,1)}\), dimension \(n-1\);
\item \(S^{(2,1^{n-2})}\) (\(\cong S \otimes S^{(n-1,1)}\)), dimension \(n-1\).
\begin{flushright}\((*)\)\end{flushright}
\end{itemize}
\end{lemma}
This is well-known, but for completeness we include a proof using the Branching Theorem and the Hook Formula.
\begin{proof}
By direct calculation using (\ref{eq:hook}), the lemma can be verified for \(n=9,10\). We proceed by induction. Assume the lemma holds for \(n-2,n-1\); we will prove it for \(n\). Let \(\alpha\) be a partition of \(n\) such that \(f^{\alpha} < {n-1 \choose 2}-1\). Consider the restriction \([\alpha] \downarrow S_{n-1}\), which has the same dimension. First suppose \([\alpha]\downarrow S_{n-1}\) is reducible. If it has one of our 4 irreducible representations (\(*\)) as a constituent, then by (\ref{eq:restriction}), the possibilies for \(\alpha\) are as follows:
\begin{center}
\begin{tabular}{l|l}
constituent & possibilies for \(\alpha\)\\
\hline
\([n-1]\) & \((n),(n-1,1)\)\\
\([1^{n-1}]\) & \((1^{n}),(2,1^{n-1})\)\\
\([n-2,1]\) & \((n-1,1),(n-2,2),(n-2,1,1)\)\\
\([2,1^{n-3}]\) & \((2,1^{n-2}),(2,2,1^{n-4}),(3,1^{n-3})\)
\end{tabular}
\end{center}

But using (\ref{eq:hook}), we see that the new irreducible representations above all have dimension at least \({n-1 \choose 2}-1\):
\begin{center}
\begin{tabular}{l|l}
\(\alpha\) & \(f^{\alpha}\)\\
\hline
\((n-2,2),(2,2,1^{n-4})\) & \({n-1 \choose 2}-1\)\\
\((n-2,1,1),(3,1^{n-3})\) & \({n-1 \choose 2}\)\\
\end{tabular}\\
\end{center}

Hence, none of these are constituents of \([\alpha] \downarrow S_{n-1}\). So we may assume that the irreducible constituents of \([\alpha] \downarrow S_{n-1}\) do not include any of our 4 irreducible representations (\(*\)), so by the induction hypothesis for \(n-1\), each has dimension at least \({n-2\choose 2}-1\). But \(2({n-2\choose 2}-1) \geq {n-1 \choose 2}-1\) provided \(n \geq 11\), so there is just one, i.e. \([\alpha] \downarrow S_{n-1}\) is irreducible. Therefore \([\alpha] = [s^{t}]\) for some \(s,t \in \mathbb{N}\) with \(st=n\), i.e. it has square Young diagram. Now consider
\[[\alpha]\downarrow S_{n-2} = [s^{t-1},s-2] + [s^{t-2},s-1,s-1].\]
Note that neither of these 2 irreducible constituents are any of our 4 irreducible representations \((*)\), so by the induction hypothesis for \(n-2\), each has dimension at least \({n-3 \choose 2}-1\). But \(2({n-3\choose 2}-1) \geq {n-1 \choose 2}-1\) for \(n \geq 11\), contradicting \(\dim ([\alpha] \downarrow S_{n-2}) < {n-1 \choose 2}-1\).
\end{proof}

If \(\alpha\) is any partition of \(n\) whose Specht module has high dimension \(f^{\alpha} \geq {n-1 \choose 2}-1\), we may bound \(|\lambda_{\alpha}|\) using the `trace method' --- specifically, we consider the trace of \(A^2\):
\begin{lemma}
\label{lemma:sumofsquaresofevals}
Let \(H\) be a graph on \(N\) vertices whose adjancency matrix \(A\) has eigenvalues \(\lambda_{1} \geq \lambda_{2} \geq \ldots \geq \lambda_{N}\); then
\[\sum_{i=1}^{N} \lambda_{i}^{2} = \textrm{Trace}(A^2) = 2e(H).\]
\end{lemma}
This is well-known; we include a proof for completeness.
\begin{proof}
Diagonalize \(A\): there exists a real invertible matrix \(P\) such that \(A = P^{-1}DP\), where \(D\) is the diagonal matrix
\[D = \left(\begin{array}{cccc} \lambda_{1} & 0 & \ldots & 0\\
0 & \lambda_{2} && 0\\
\vdots & & \ddots & \vdots\\
0 & & \ldots & \lambda_{N} \end{array} \right).\] 
We have \(A^{2} = P^{-1}D^{2}P\), and therefore
\[2e(H) = \sum_{i,j=1}^{N}A_{i,j} = \sum_{i,j=1}^{N} A_{i,j}^{2} = \textrm{Trace}(A^{2}) = \textrm{Trace}(P^{-1}D^{2}P) = \textrm{Trace}(D^{2}) = \sum_{i=1}^{N}\lambda_{i}^{2},\]
as required.
\end{proof}
Hence, the eigenvalues of the derangement graph satisfy:
\[\sum_{\alpha \vdash n} (f^{\alpha}\lambda_{\alpha})^{2} = 2e(\Gamma) = n! d_{n} = (n!)^{2} (1/e + o(1)),\]
so for each partition \(\alpha\) of \(n\),
\[|\lambda_{\alpha}| \leq \frac{\sqrt{n! d_{n}}}{f^{\alpha}} = \frac{n!}{f^{\alpha}}\sqrt{1/e+o(1)}.\]
Let
\[\mathcal{M} = \left\{\alpha \vdash n:\ f^\alpha \geq {n-1 \choose 2}+1\right\};\]
we have
\[\max_{\alpha \in \mathcal{M}}|\lambda_{\alpha}| \leq O((n-2)!).\]
For each of the other Specht modules \((*)\), we now explicitly calculate the corresponding eigenvalue using (\ref{eq:evalformula}).

For the trivial module, \(\chi_{(n)} \equiv 1\), so
\[\lambda_{(n)} = d_{n}.\]

For the sign module \(S^{(1^n)}\), \(\chi_{(1^{n})} = \epsilon\), so
\[\lambda_{(1^{n})} = \sum_{\sigma \in \mathcal{D}_{n}} \epsilon(\sigma) = e_{n}-o_{n}\]
where \(e_{n},o_{n}\) are the number of even and odd derangements of \([n]\), respectively. It is well known that for any \(n \in \mathbb{N}\),
\begin{equation}
\label{eq:derangements}
e_{n}-o_{n} = (-1)^{n-1}(n-1).
\end{equation}
To see this, note that an odd permutation \(\sigma \in S_{n}\) without fixed points can be written as \((i\ n)\rho\), where \(\sigma(n)=i\), and \(\rho\) is either an even permutation of \([n-1]\setminus \{i\}\) with no fixed points (if \(\sigma(i)=n\)), or an even permutation of \([n-1]\) with no fixed points (if \(\sigma(i) \neq n\)). Conversely, for any \(i \neq n\), if \(\rho\) is any even permutation of \([n-1]\) with no fixed points or any even permutation of \([n-1]\setminus \{i\}\) with no fixed points, then \((i\ n) \rho\) is a permutation of \([n]\) with no fixed points taking \(n \mapsto i\). Hence, for all \(n \geq 3\),
\[o_{n} = (n-1)(e_{n-1}+e_{n-2}).\]
Similarly,
\[e_{n} = (n-1)(o_{n-1}+o_{n-2}).\]
Equation (\ref{eq:derangements}) follows by induction on \(n\).

Hence, we have:
\[\lambda_{(1^{n})} = (-1)^{n-1}(n-1).\]

For the partition \((n-1,1)\), from (\ref{eq:permcharacter}) we have:
\[\chi_{(n-1,1)}(\sigma) = \xi_{(n-1,1)}(\sigma)-1 = \#\{\textrm{fixed points of } \sigma\} -1,\]
so we obtain
\[\lambda_{(n-1,1)} = \frac{1}{n-1}\sum_{\sigma \in \mathcal{D}_{n}}(-1) = -\frac{d_{n}}{n-1}.\]

For \(S^{(2,1^{n-2})} \cong S^{(1^{n})} \otimes S^{(n-1,1)}\), we have \(\chi_{(2,1^{n-2})} = \epsilon \cdot \chi_{(n-1,1)}\), so
\[\chi_{(2,1^{n-2})}(\sigma) = \epsilon(\sigma)(\#\{\textrm{fixed points of }\sigma\}-1),\]
and therefore
\[\lambda_{(2,1^{n-2})} = \frac{1}{n-1}\sum_{\sigma \in \mathcal{D}_{n}} -\epsilon(\sigma) = -\frac{e_{n}-o_{n}}{n-1} = (-1)^{n}.\]

To summarize, we obtain:\\
\begin{center}

\begin{tabular}{l|l}
\(\alpha\) & \(\lambda_{\alpha}\)\\
\hline
\((n)\) & \(d_{n}\)\\
\((1^{n})\) & \((-1)^{n-1}(n-1)\)\\
\((n-1,1)\) & \(-d_{n}/(n-1)\)\\
\((2,1^{n-2})\) & \((-1)^{n}\)
\end{tabular}\\
\end{center}
\vspace{\baselineskip}
Hence, \(U_{(n)}\) is the \(d_{n}\)-eigenspace, \(U_{(n-1,1)}\) is the \(-d_{n}/(n-1)\)-eigenspace, and all other eigenvalues are \(O((n-2)!)\). Hence, Leader's conjecture follows (for \(n\) sufficiently large) by applying Theorem \ref{thm:crosshoffman} to the derangement graph. It is easy to check that \(\nu = d_{n}/(n-1)\) for all \(n \geq 4\), giving

\begin{theorem}
\label{thm:leaderconj}
For \(n \geq 4\), if \(\mathcal{A},\mathcal{B} \subset S_{n}\) are cross-intersecting, then
\[|\mathcal{A}||\mathcal{B}| \leq ((n-1)!)^{2}.\]
\end{theorem}

If equality holds, then by Theorem \ref{thm:crosshoffman} part \textit{(ii)}, the characteristic vectors \(v_{\mathcal{A}},v_{\mathcal{B}}\) must lie in the direct sum of the \(d_{n}\) and \(-d_{n}/(n-1)\)-eigenspaces. It can be checked that for \(n \geq 5\), \(|\lambda_{\alpha}| < d_{n}/(n-1)\ \forall \alpha \neq (n),(n-1,1)\), so the \(d_{n}\) eigenspace is precisely \(U_{(n)}\) and the \(-d/(n-1)\)-eigenspace is precisely \(U_{(n-1,1)}\). But we have:

\begin{lemma}
For \(i,j \in [n]\), let \(v_{i \mapsto j} = v_{\{\sigma \in S_{n}:\ \sigma(i)=j\}} \in \mathbb{C}S_n\) be the characteristic vector of the 1-coset \(\{\sigma \in S_{n}:\ \sigma(i)=j\}\). Then
\[U_{(n)} \oplus U_{(n-1,1)} = \textrm{Span}\{v_{i \mapsto j}\ : i,j \in [n]\}.\]
\end{lemma}
This is a special case of Theorem 7 in \cite{jointpaper}. We give a short proof for completeness.
\begin{proof}
Let
\[U=\textrm{Span}\{v_{i \mapsto j}\ : i,j \in [n]\}.\]
For each \(i \in [n]\), \(\{v_{i,j}\ : j \in [n]\}\) is a basis for a copy \(W_{i}\) of the permutation module \(M^{(n-1,1)}\) in \(\mathbb{C}S_{n}\). Since
\[M^{(n-1,1)} \cong S^{(n)}\oplus S^{(n-1,1)},\]
we have the decomposition
\[W_{i} = \textrm{Span}\{\boldsymbol{f}\} \oplus V_{i},\]
where \(V_{i}\) is some copy of \(S^{(n-1,1)}\) in \(\mathbb{C}S_{n}\), so
\[\textrm{Span}\{v_{i \mapsto j}\ : j \in [n]\} = W_{i} \leq U_{(n)} \oplus U_{(n-1,1)}\]
for each \(i \in [n]\), and therefore \(U \leq U_{(n)} \oplus U_{(n-1,1)}\).

It is well known that if $G$ is any finite group, and $T,T'$ are two isomorphic submodules of $\mathbb{C}G$, then there exists $s \in \mathbb{C}G$ such that the right multiplication map $x \mapsto xs$ is an isomorphism from $T$ to $T'$ (see for example \cite{jamesliebeck}). Hence, for any \(i \in [n]\), the sum of all right translates of $W_{i}$ contains \(\textrm{Span}\{\boldsymbol{f}\}\) and all submodules of $\mathbb{C}S_{n}$
isomorphic to $S^{(n-1,1)}$, so $U_{(n)} \oplus U_{(n-1,1)} \leq U$. Hence, $U = U_{(n)} \oplus U_{(n-1,1)}$ as required.
\end{proof}

Hence, for \(n \geq 5\), if equality holds in Theorem \ref{thm:leaderconj}, then the characteristic vectors of \(\mathcal{A}\) and \(\mathcal{B}\) are linear combinations of the characteristic vectors of the \(1\)-cosets. It was proved in \cite{jointpaper} that if the characteristic vector of \(\mathcal{A} \subset S_{n}\) is a linear combination of the characteristic vectors of the 1-cosets, then \(\mathcal{A}\) is a disjoint union of 1-cosets. It follows that for \(n \geq 5\), if equality holds in Theorem \ref{thm:leaderconj}, then \(\mathcal{A}\) and \(\mathcal{B}\) are both disjoint unions of 1-cosets. Since they are cross-intersecting, they must both be equal to the same 1-coset, i.e.
\[\mathcal{A}=\mathcal{B} = \{\sigma \in S_{n}:\ \sigma(i)=j\}\]
for some \(i,j \in [n]\). It is easily checked that the same conclusion holds when \(n=4\), so we have the following characterization of the case of equality in Leader's conjecture:

\begin{theorem}
For \(n \geq 4\), if \(\mathcal{A},\mathcal{B} \subset S_{n}\) are cross-intersecting and satisfy
\[|\mathcal{A}||\mathcal{B}| = ((n-1)!)^{2},\]
then
\[\mathcal{A}=\mathcal{B} = \{\sigma \in S_{n}:\ \sigma(i)=j\}\]
for some \(i,j \in [n]\).
\end{theorem}

\section{Stability}

We will now perform a stability analysis for intersecting families of permutations. First, we prove a `rough' stability result: for any positive constant \(c > 0\), if \(\mathcal{A} \subset S_n\) is an intersecting family of permutations with \(|\mathcal{A}|\geq c(n-1)!\), then there exist \(i\) and \(j\) such that all but \(O((n-2)!)\) permutations in \(\mathcal{A}\) map \(i\) to \(j\), i.e. \(\mathcal{A}\) is `almost' centred. In other words, writing \(\mathcal{A}_{i \mapsto j}\) for the collection of all permutations in \(\mathcal{A}\) mapping \(i\) to \(j\), we have \(|\mathcal{A} \setminus \mathcal{A}_{i \mapsto j}| \leq O((n-2)!)\). To prove this, our first step will be to show that if \(\mathcal{A} \subset S_n\) is an intersecting family with \(|\mathcal{A}| \geq c(n-1)!\), then the characteristic vector \(v_{\mathcal{A}}\) of \(\mathcal{A}\) cannot be too far from the subspace \(U\) spanned by the characteristic vectors of the 1-cosets, the intersecting families of maximum size. Secondly, we will use this to show that there exist \(i,j \in [n]\) such that \(|\mathcal{A}_{i \mapsto j}| \geq \omega((n-2)!)\). Clearly, for any fixed \(i \in [n]\),
\[\sum_{k=1}^{n}|\mathcal{A}_{i \mapsto k}| = |\mathcal{A}|,\]
and therefore the average size of an \(|\mathcal{A}_{i \mapsto k}|\) is \(|\mathcal{A}| / n\); we have found \(i\) and \(j\) such that \(|\mathcal{A}_{i \mapsto j}|\) is \(\omega\) of the average size. This statement would at first seem too weak to help us, but using the fact that \(\mathcal{A}\) is intersecting, we will `bootstrap' it to the much stronger statement \(|\mathcal{A}_{i \mapsto j}| \geq (1-o(1))|\mathcal{A}|\). In detail, we will deduce from Theorem \ref{thm:leaderconj} that for any \(k \neq j\),
\[|\mathcal{A}_{i \mapsto j}||\mathcal{A}_{i \mapsto k}| \leq ((n-2)!)^{2},\]
giving \(|\mathcal{A}_{i \mapsto k}| \leq o((n-2)!)\) for any \(k \neq j\). Summing over all \(k \neq j\) will give \(|\mathcal{A} \setminus \mathcal{A}_{i \mapsto j}| \leq o((n-1)!)\), enabling us to complete the proof.

Note that this is enough to prove the stability conjecture of Cameron and Ku: if \(\mathcal{A}\) is non-centred, it must contain some permutation \(\tau\) such that \(\tau(i) \neq j\). This immediately forces \(|\mathcal{A}_{i \mapsto j}|\) to be less than \((1-1/e+o(1))(n-1)!\), yielding a contradiction if \(c > 1-1/e\), and \(n\) is sufficiently large depending on \(c\).

Here, then, is our rough stability result:

\begin{theorem}
\label{thm:roughstability}
For any \(c > 0\), there exists \(K >0\) such that the following holds. If \(\mathcal{A}\subset S_{n}\) is an intersecting family of permutations with \(|\mathcal{A}|\geq c(n-1)!\), then there exist \(i,j \in [n]\) such that
\[|\mathcal{A} \setminus \mathcal{A}_{i \mapsto j}| \leq K(n-2)!.\]
\end{theorem}
To carry out the first of the above steps, we will need a `stability' version of Hoffman's theorem:
\begin{lemma}
\label{lemma:stabilityhoffman}
Let \(H = (V,E)\) be a \(d\)-regular graph on \(N\) vertices, and let \(A\) denote the adjacency matrix of \(H\). Let \(\lambda_{N}\) denote the least eigenvalue of \(A\), and let \(U = \textrm{Span}(\mathbf{f}) \oplus \textrm{Ker}(A - \lambda_N I)\). Let
\(\lambda_M = \min\{\lambda_i:\ \lambda_i \neq \lambda_N\}\). Let \(X \subset V(H)\) be an independent set in \(H\), and let \(\alpha = |X|/N\) denote its measure. Equip \(\mathbb{R}^{V}\) with the inner product:
\[\langle u,v \rangle = \frac{1}{N}\sum_{i=1}^{N} u(i) v(i),\]
and let
\[||u|| = \sqrt{\frac{1}{N}\sum_{i=1}^{N}u(i)^{2}}\]
be the induced Euclidean norm. Let \(D\) denote the Euclidean distance from the characteristic vector \(v_{X}\) of \(X\) to the subspace \(U\), i.e. the norm \(||P_{U^{\perp}}(v_{X})||\) of the projection of \(v_{X}\) onto \(U^{\perp}\). Then
\[D^{2} \leq \frac{(1-\alpha)|\lambda_{N}| - \lambda_{1} \alpha}{|\lambda_{N}|-|\lambda_{M}|}\alpha.\]
\end{lemma}
\begin{proof}
This is a straightforward adaptation of the proof of Hoffman's theorem. Let \(u_{1} = \mathbf{f},u_{2},\ldots,u_{N}\) be an orthonormal basis of real eigenvectors of \(A\) corresponding to the eigenvalues \(\lambda_{1}=d,\lambda_{2},\ldots,\lambda_{N}\). Write
\[v_{X} = \sum_{i=1}^{N} \xi_{i}u_{i}\]
as a linear combination of these eigenvectors. We have \(\xi_{1}=\alpha\) and
\[\sum_{i=1}^{N}\xi_{i}^{2} = ||v_{X}||^{2} = \alpha.\]
Since \(X\) is an independent set in \(H\), we have:
\[0=\sum_{x,y \in X}A_{x,y}=v_{X}^{\top} A v_{X} = \sum_{i=1}^{N} \lambda_{i} \xi_{i}^{2} \geq d\xi_{1}^{2} + \lambda_{N} \sum_{i:\lambda_{i}=\lambda_{N}} \xi_{i}^{2} + \lambda_{M} \sum_{i>1:\lambda_{i} \neq \lambda_{N}}\xi_{i}^{2}.\]
Note that
\[\sum_{i>1:\lambda_{i} \neq \lambda_{N}}\xi_{i}^{2} = D^{2}\]
and
\[\sum_{i:\lambda_{i}=\lambda_{N}} \xi_{i}^{2} = \alpha - \alpha^{2} - D^{2},\] 
so we have
\[0 \geq d \alpha^{2} + \lambda_{N} (\alpha -\alpha^{2} -D^{2}) + \lambda_{M} D^{2}.\]
Rearranging, we obtain:
\[D^{2} \leq \frac{(1-\alpha)|\lambda_{N}| - d \alpha}{|\lambda_{N}|-|\lambda_{M}|}\alpha,\]
as required.
\end{proof}
For the second step, we will need an isoperimetric inequality for the {\em transposition graph} on \(S_n\). If \(H = (V,E)\) is a graph, and \(x,y \in V\), we define the {\em graph distance} \(d_{H}(x,y)\) to be the length of the shortest path in \(H\) between \(x\) and \(y\). If \(X \subset V(H)\), and \(h >0\), we define the {\em \(h\)-neighbourhood} \(N_{h}(X)\) to be the set of vertices of \(H\) which are at distance at most \(h\) from \(X\), i.e.
\[N_{h}(X) = \{y \in V:\ d_{H}(x,y) \leq h\ \textrm{for some } x \in X\}.\]

The \textit{transposition graph} \(T\) is the Cayley graph on \(S_{n}\) generated by the transpositions, i.e. \(V(T) = S_{n}\) and \(\sigma \pi \in E(T)\) if and only if \(\sigma^{-1} \pi\) is a transposition. We will use the following isoperimetric inequality for \(T\), essentially the martingale inequality of Maurey:
\begin{theorem} \label{thm:maurey}
Let \(0 < a < 1\), and let \(X \subset V(T)\) with \(|X| \geq a n!\). Then for any \(h \geq h_{0} := \sqrt{\tfrac{1}{2}(n-1)\log \tfrac{1}{a}}\),
\[|N_{h}(X)| \geq \left(1-e^{-\frac{2(h-h_{0})^{2}}{n-1}}\right)n!.\]
\end{theorem}
(For a proof, see for example \cite{McD}.) We will deduce from this that for any two sets \(X,Y \subset S_n\) which are not too small, there exist permutations \(\sigma \in X\) and \(\tau \in Y\) which are `close' to one another in \(T\).

Finally, we need the following simple consequence of Theorem \ref{thm:leaderconj}:
\begin{lemma}
\label{lemma:crosslemma}
Let \(\mathcal{A} \subset S_{n}\) be an intersecting family; then for all \(i, j\) and \(k\) with \(k \neq j\),
\[|\mathcal{A}_{i\mapsto j}||\mathcal{A}_{i \mapsto k}| \leq ((n-2)!)^{2}.\]
\end{lemma}
\begin{proof}
By double translation, we may assume that \(i=j=1\) and \(k=2\). Let \(\sigma \in \mathcal{A}_{1 \mapsto 1}\) and \(\pi \in \mathcal{A}_{1 \mapsto 2}\); then there exists \(p \neq 1\) such that \(\sigma(p) = \pi(p) > 2\). Hence, the translates \(\mathcal{E} = \mathcal{A}_{1 \mapsto 1}\) and \(\mathcal{F} = (1 \ 2)\mathcal{A}_{1 \mapsto 2}\) are families of permutations fixing 1 and cross-intersecting on the domain \(\{2,3,\ldots,n\}\). Deleting 1 from each permutation in the two families gives a cross-intersecting pair \(\mathcal{E}',\mathcal{F}'\) of families of permutations of \(\{2,3,\ldots,n\}\); applying Theorem \ref{thm:leaderconj} gives:
\[|\mathcal{A}_{1 \mapsto 1}||\mathcal{A}_{1 \mapsto 2}| = |\mathcal{E}'||\mathcal{F}'| \leq ((n-2)!)^{2}.\]
\end{proof}

\begin{proof}[Proof of Theorem \ref{thm:roughstability}:]
Let \(c > 0\) be a positive constant, and let \(\mathcal{A}\subset S_{n}\) be an intersecting family of permutations with \(|\mathcal{A}|\geq c(n-1)!\). Write \(\alpha = |\mathcal{A}|/n!\). Since \(\mathcal{A}\) is an independent set in the derangement graph \(\Gamma\), which has \(|\lambda_M| = O((n-2)!)\), Lemma \ref{lemma:stabilityhoffman} yields:
\begin{eqnarray*}
D^{2} & \leq & \frac{(1-\alpha)d_{n}/(n-1) -d_{n}\alpha}{d_{n}/(n-1)-|\lambda_{M}|}\frac{|\mathcal{A}|}{n!}\\
& = & \frac{1-\alpha -\alpha(n-1)}{1-(n-1)|\lambda_{M}|/d_{n}}\frac{|\mathcal{A}|}{n!}\\
& = & \frac{1-\alpha n}{1-O(1/n)}\frac{|\mathcal{A}|}{n!}\\
& = & (1-\alpha n)(1+O(1/n))|\mathcal{A}|/n!,
\end{eqnarray*}
where \(D = ||P_{U^{\perp}}(v_{\mathcal{A}})||\) denotes the Euclidean distance from \(v_{\mathcal{A}}\) to the subspace
\[U = U_{(n)} \oplus U_{(n-1,1)} = \textrm{Span}\{v_{i \mapsto j}\ : i,j \in [n]\}.\]
Write \(|\mathcal{A}| = (1-\delta)(n-1)!\), where \(\delta < 1\). Then
\begin{equation}
\label{eq:perpprojection}
||P_{U^{\perp}}(v_{\mathcal{A}})||^{2}  = D^2 \leq \delta(1 + O(1/n))|\mathcal{A}|/n!.
\end{equation}
We now derive a formula for \(P_{U}(v_{A})\). The projection of \(v_{\mathcal{A}}\) onto \(U_{(n)} = \textrm{Span}\{\mathbf{f}\}\) is clearly \((|\mathcal{A}|/n!)\mathbf{f}\). By (\ref{eq:idempotents}), the primitive central idempotent generating \(U_{(n-1,1)}\) is
\[\frac{n-1}{n!} \sum_{\pi \in S_{n}} \chi_{(n-1,1)}(\pi^{-1}) \pi,\]
and therefore the projection of \(v_{\mathcal{A}}\) onto \(U_{(n-1,1)}\) is given by
\[P_{U_{(n-1,1)}}(v_{\mathcal{A}}) = \frac{n-1}{n!} \sum_{\rho \in \mathcal{A}} \sum_{\pi \in S_{n}} \chi_{(n-1,1)}(\pi^{-1}) \pi \rho,\]
which has \(\sigma\)-coordinate
\begin{eqnarray*}
P_{U_{(n-1,1)}}(v_{\mathcal{A}})_{\sigma} & =& \frac{n-1}{n!}\sum_{\rho \in \mathcal{A}} \chi_{(n-1,1)}(\rho \sigma^{-1})\\
&=& \frac{n-1}{n!}\sum_{\rho \in \mathcal{A}}(\xi_{(n-1,1)}(\rho \sigma^{-1}) - 1)\\
& = & \frac{n-1}{n!} \sum_{\rho \in \mathcal{A}} (\#\{\textrm{fixed points of }\rho \sigma^{-1}\}-1)\\
& = & \frac{n-1}{n!}(\#\{(\rho,i): \rho \in \mathcal{A},\ i \in [n],\ \rho(i)=\sigma(i)\} - |\mathcal{A}|)\\
& = & \frac{n-1}{n!} \sum_{i=1}^{n} |\mathcal{A}_{i \mapsto \sigma(i)}| - \frac{n-1}{n!}|\mathcal{A}|.
\end{eqnarray*}
Hence, the \(\sigma\)-coordinate \(P_{\sigma}\) of the projection of \(v_{\mathcal{A}}\) onto \(U = U_{(n)}\oplus U_{(n-1,1)}\) is given by
\[P_{\sigma} = \frac{n-1}{n!} \sum_{i=1}^{n} |\mathcal{A}_{i \mapsto \sigma(i)}| - \frac{(n-2)}{n!}|\mathcal{A}|,\]
which is a linear function of
\[\sum_{i=1}^{n} |\mathcal{A}_{i \mapsto \sigma(i)}| = \#\{(\rho,i) \in  \mathcal{A} \times [n] :\ \rho(i)=\sigma(i)\},\]
the number of times \(\sigma\) agrees with a permutation in \(\mathcal{A}\).

From (\ref{eq:perpprojection}), we have
\[\sum_{\sigma \in \mathcal{A}} (1-P_{\sigma})^{2}+\sum_{\sigma \notin \mathcal{A}} P_{\sigma}^{2} \leq |\mathcal{A}|\delta (1+O(1/n)).\]
Choose \(C > 0\) such that \(|\mathcal{A}|(1-1/n)\delta(1+C/n)\) is at least the right-hand side; then
\[(1-P_{\sigma})^{2} < \delta(1+C/n)\]
for at least \(|\mathcal{A}|/n\) permutations in \(\mathcal{A}\), so the subset
\[\mathcal{A}' := \{\sigma \in \mathcal{A}: (1-P_{\sigma})^{2} < \delta(1+C/n)\}\]
has
\begin{equation}
\label{eq:boundA'}
|\mathcal{A}'| \geq |\mathcal{A}|/n.
\end{equation}
Similarly, \(P_{\sigma}^{2} < 2\delta/n\) for all but at most
\[n|\mathcal{A}|(1+O(1/n))/2 = (1-\delta)n!(1+O(1/n))/2\]
permutations \(\sigma \notin \mathcal{A}\), so the subset \(\mathcal{R} = \{\sigma \notin \mathcal{A}: P_{\sigma}^{2} < 2\delta/n\}\) has
\begin{equation}\label{eq:boundT}
 |\mathcal{R}| \geq n! - (1-\delta)(n-1)!-(1-\delta)n!(1+O(1/n))/2.
\end{equation}
The permutations \(\sigma \in \mathcal{A}'\) have \(P_{\sigma}\) close to 1; the permutations \(\pi \in \mathcal{R}\) have \(P_{\pi}\) close to 0. Using only the lower bounds (\ref{eq:boundA'}) and (\ref{eq:boundT}) on the sizes of \(\mathcal{A}'\) and \(\mathcal{R}\), we may prove the following:\\
\\
\textit{Claim:} There exist permutations \(\sigma \in \mathcal{A}',\ \pi \in \mathcal{R}\) such that \(\sigma^{-1} \pi\) is a product of at most \(h=h(n)\) transpositions, where \(h = 2\sqrt{2 (n-1) \log n}\).\\
\\
\textit{Proof of Claim:} Apply Theorem \ref{thm:maurey} to the set \(\mathcal{A}'\), with \(a = 1/n^{4}\) and \(h = 2h_{0}\). Since \(|\mathcal{A}'| \geq \frac{c(n-1)!}{n} \geq \frac{n!}{n^{4}}\), we have
\[|N_{h}(\mathcal{A}')| \geq (1-n^{-4})n!,\]
so certainly \(N_{h}(\mathcal{A}') \cap \mathcal{R} \neq \emptyset\), proving the claim.

We now have two permutations \(\sigma \in \mathcal{A}\), \(\pi \notin \mathcal{A}\) which are `close' to one another in \(T\) (differing in only \(O(\sqrt{n \log n})\) transpositions) such that \(P_{\sigma} > 1-\sqrt{\delta(1+C/n)}\) and \(P_{\pi} < \sqrt{2\delta/n}\), and therefore \(P_{\sigma}-P_{\pi} > 1-\sqrt{\delta}-O(1/\sqrt{n})\), i.e. \(\sigma\) agrees many more times than \(\pi\) with permutations in \(\mathcal{A}\):
\[\sum_{i=1}^{n}|\mathcal{A}_{i \mapsto \sigma(i)}| - \sum_{i=1}^{n}|\mathcal{A}_{i \mapsto \pi(i)}| \geq (n-1)!(1-\sqrt{\delta}-O(1/\sqrt{n})).\]
Suppose for this pair we have \(\pi = \sigma \tau_{1} \tau_{2} \ldots \tau_{l}\) for transpositions \(\tau_{1},\ldots, \tau_{l}\), where \(l \leq t\). Let \(I\) be the set of numbers appearing in these transpositions; then \(|I| \leq 2l \leq 2t\), and \(\sigma(i)=\pi(i)\) for each  \(i \notin I\). Hence,
\[\sum_{i\in I}|\mathcal{A}_{i\mapsto \sigma(i)}| - \sum_{i\in I}|\mathcal{A}_{i \mapsto \pi(i)}| \geq (n-1)!(1-\sqrt{\delta}-O(1/\sqrt{n})),\]
so certainly,
\[\sum_{i\in I}|\mathcal{A}_{i\mapsto \sigma(i)}| \geq (n-1)!(1-\sqrt{\delta}-O(1/\sqrt{n})).\]
By averaging,
\begin{eqnarray*}
|\mathcal{A}_{i \mapsto \sigma(i)}|& \geq& \frac{1}{|I|} (n-1)!(1-\sqrt{\delta}-O(1/\sqrt{n}))\\
& \geq & \frac{(n-1)!}{4\sqrt{2 (n-1) \log n}} (1-\sqrt{\delta}-O(1/\sqrt{n}))
\end{eqnarray*}
for some \(i \in I\). Let \(\sigma(i)=j\); then
\[|\mathcal{A}_{i \mapsto j}| \geq \frac{(n-1)!}{4\sqrt{2 (n-1) \log n}} (1-\sqrt{1-c}-O(1/\sqrt{n})) = \omega((n-2)!).\]

It follows from Lemma \ref{lemma:crosslemma} that \(|\mathcal{A}_{i \mapsto k}| \leq o((n-2)!)\) for all \(k \neq j\). Summing over all \(k \neq j\) gives
\[|\mathcal{A} \setminus \mathcal{A}_{i \mapsto j}| = \sum_{k \neq j} |\mathcal{A}_{i \mapsto k}| \leq o((n-1)!),\]
and therefore
\begin{equation}
\label{eq:bound11}
|\mathcal{A}_{i \mapsto j}| = |\mathcal{A}|-|\mathcal{A} \setminus \mathcal{A}_{i\mapsto j}| \geq (c-o(1))(n-1)!.
\end{equation}
Applying Lemma \ref{lemma:crosslemma} again gives
\[|\mathcal{A}_{i \mapsto k}| \leq O((n-3)!)\]
for all \(k \neq j\); summing over all \(k \neq j\) gives
\[|\mathcal{A} \setminus \mathcal{A}_{i \mapsto j}| \leq O((n-2)!),\]
proving Theorem \ref{thm:roughstability}.
\end{proof}
The stability conjecture of Cameron and Ku follows easily:
\begin{corollary}
Let \(c > 1-1/e\); then for \(n\) sufficiently large depending on \(c\), any intersecting family \(\mathcal{A} \subset S_{n}\) with \(|\mathcal{A}| \geq c(n-1)!\) is centred.
\end{corollary}
\begin{proof}
Let \(c > 1-1/e\), and let \(\mathcal{A} \subset S_n\) be intersecting, with \(|\mathcal{A}| \geq c(n-1)!\). By Theorem \ref{thm:roughstability}, there exist \(i,j \in [n]\) such that \(|\mathcal{A}\setminus \mathcal{A}_{i \mapsto j}| \leq O((n-2)!)\), and therefore
\begin{equation}
\label{eq:oneislarge}
|\mathcal{A}_{i \mapsto j}| \geq (c - O(1/n))(n-1)!.
\end{equation}
Suppose for a contradiction that \(\mathcal{A}\) is non-centred. Then there exists a permutation \(\tau \in \mathcal{A}\) such that \(\tau(i) \neq j\). Any permutation in \(\mathcal{A}_{i \mapsto j}\) must agree with \(\tau\) at some point. But for any \(i,j \in [n]\) and any \(\tau \in S_{n}\) such that \(\tau(i) \neq j\), the number of permutations in \(S_{n}\) which map \(i\) to \(j\) and agree with \(\tau\) at some point is
\[(n-1)!-d_{n-1}-d_{n-2} = (1-1/e-o(1))(n-1)!.\]
(By double translation, we may assume that \(i=j=1\) and \(\tau = (1\ 2)\); we observed above that the number of permutations fixing \(1\) and intersecting \((1\ 2)\) is \((n-1)!-d_{n-1}-d_{n-2}\).) This contradicts (\ref{eq:oneislarge}) provided \(n\) is sufficiently large depending on \(c\).
\end{proof}

We now use our rough stability result to prove the Hilton-Milner type conjecture of Cameron and Ku, for \(n\) sufficiently large. First, we introduce an extra notion which will be useful in the proof. Following Cameron and Ku \cite{cameron}, given a permutation \(\pi \in S_{n}\) and \(i \in [n]\), we define the \(i\)-\textit{fix} of \(\pi\) to be the permutation \(\pi_{i}\) which fixes \(i\), maps the preimage of \(i\) to the image of \(i\), and agrees with \(\pi\) at all other points of \([n]\), i.e.
\[\pi_{i}(i) = i;\ \pi_{i}(\pi^{-1}(i)) = \pi(i);\ \pi_{i}(k)=\pi(k)\ \forall k \neq i,\pi^{-1}(i).\]
In other words, \(\pi_{i} = \pi(\pi^{-1}(i)\ i)\). We inductively define
\[\pi_{i_{1},\ldots,i_{l}} = (\pi_{i_{1},\ldots,i_{l-1}})_{i_{l}}.\]
Notice that if \(\sigma\) fixes \(j\), then \(\sigma\) agrees with \(\pi_{j}\) wherever it agrees with \(\pi\).

\begin{theorem}
\label{thm:cameronkuconj}
For \(n\) sufficiently large, if \(\mathcal{A} \subset S_{n}\) is a non-centred intersecting family, then \(\mathcal{A}\) is at most as large as the family
\[\mathcal{C} = \{\sigma \in S_{n}: \sigma(1)=1, \sigma(i)=i \textrm{ for some } i > 2\} \cup \{(12)\},\]
which has size \((n-1)!-d_{n-1}-d_{n-2}+1 = (1-1/e+o(1))(n-1)!\). Equality holds if and only if \(\mathcal{A}\) is a double translate of \(\mathcal{C}\), i.e. \(\mathcal{A} = \pi \mathcal{C} \tau\) for some \(\pi,\tau \in S_{n}\).
\end{theorem}
\begin{proof}
Let \(\mathcal{A} \subset S_n\) be a non-centred intersecting family with the same size as \(\mathcal{C}\); we must show that \(\mathcal{A}\) is a double translate of \(\mathcal{C}\). By Theorem \ref{thm:roughstability}, there exist \(i,j \in [n]\) such that \(|\mathcal{A}\setminus \mathcal{A}_{i \mapsto j}| \leq O((n-2)!)\), and therefore
\[|\mathcal{A}_{i \mapsto j}| \geq (n-1)!-d_{n-1}-d_{n-2}+1-O(n-2)! = (1-1/e-o(1))(n-1)!.\]
Since \(\mathcal{A}\) is non-centred, it must contain some permutation \(\rho\) such that \(\rho(i) \neq j\). By double translation, we may assume that \(i=j=1\) and \(\rho = (1\ 2)\); we will show that under these hypotheses, \(\mathcal{A} = \mathcal{C}\). We have
\begin{equation}
\label{eq:A11islarge}
|\mathcal{A}_{1 \mapsto 1}| \geq (1-1/e-o(1))(n-1)!
\end{equation}
and \((1 \ 2) \in \mathcal{A}\). Note that every permutation in \(\mathcal{A}\) must intersect \((1\ 2)\), and therefore
\[\mathcal{A}_{1 \mapsto 1} \cup \{(1 \ 2)\} \subset \mathcal{C}.\]
We need to show that \((1\ 2)\) is the only permutation in \(\mathcal{A}\) that does not fix 1. Suppose for a contradiction that \(\mathcal{A}\) contains some other permutation \(\pi\) not fixing 1. Then \(\pi\) must shift some point \(p > 2\). If \(\sigma\) fixes both 1 and \(p\), then \(\sigma\) agrees with \(\pi_{1,p} = (\pi_{1})_{p}\) wherever it agrees with \(\pi\). There are exactly \(d_{n-2}\) permutations which fix 1 and \(p\) and disagree with \(\pi_{1,p}\) at every point of \(\{2,\ldots,n\}\setminus \{p\}\); each disagrees everywhere with \(\pi\), so none are in \(\mathcal{A}\), and therefore
\[|\mathcal{A}_{1 \mapsto 1}| \leq (n-1)!-d_{n-1}-2d_{n-2}.\]
Hence, by assumption,
\[|\mathcal{A} \setminus \mathcal{A}_{1 \mapsto 1}| \geq d_{n-2}+1 = \Omega((n-2)!).\]

Notice that we have the following trivial bound on the size of a \(t\)-intersecting family\footnote{We say that a family \(\mathcal{F} \subset S_n\) is {\em \(t\)-intersecting} if any two permutations in \(\mathcal{F}\) agree on at least \(t\) points.} \(\mathcal{F} \subset S_{n}\):
\[|\mathcal{F}| \leq {n \choose t}(n-t)! = n!/t!\]
since every permutation in \(\mathcal{F}\) must agree with a fixed \(\rho \in \mathcal{F}\) in at least \(t\) places.

Hence, \(\mathcal{A}\setminus \mathcal{A}_{1\mapsto 1}\) cannot be \((\log n)\)-intersecting and therefore contains two permutations \(\rho,\tau\) agreeing on at most \(\log n\) points. The number of permutations fixing \(1\) and agreeing with both \(\tau_{1}\) and \(\tau_{2}\) at one of these points is at most \((\log n)(n-2)!\). All other permutations in \(\mathcal{A} \cap \mathcal{C}\) agree with \(\rho\) and \(\tau\) at two separate points of \(\{2,\ldots,n\}\), and by the above argument, the same holds for the 1-fixes \(\rho_{1}\) and \(\tau_{1}\). The number of permutations fixing 1 that agree with \(\rho_{1}\) and \(\tau_{1}\) at two separate points of \(\{2,\ldots,n\}\) is at most \(((1-1/e)^{2} + o(1))(n-1)!\) (it is easily checked that given two fixed permutations, the probability that a uniform random permutation agrees with them at separate points is at most \((1-1/e)^{2} + o(1)\)). Hence,
\begin{eqnarray*}
|\mathcal{A}_{1 \mapsto 1}| & \leq & ((1-1/e)^{2} + o(1))(n-1)! + (\log n) (n-2)!\\
& = & ((1-1/e)^{2}+o(1))(n-1)!,
\end{eqnarray*}
contradicting (\ref{eq:A11islarge}) provided \(n\) is sufficiently large.

Hence, \((1 \ 2)\) is the only permutation in \(\mathcal{A}\) that does not fix 1, so \(\mathcal{A} = \mathcal{A}_{1 \mapsto 1} \cup \{(1\ 2)\} \subset \mathcal{C}\); since \(|\mathcal{A}|=|\mathcal{C}|\), we have \(\mathcal{A}=\mathcal{C}\) as required.
\end{proof}
We now perform a very similar stability analysis for cross-intersecting families. First, we prove a `rough' stability result analogous to Theorem \ref{thm:roughstability}, namely that for any positive constant \(c > 0\), if \(\mathcal{A},\mathcal{B} \subset S_{n}\) are cross-intersecting with \(\sqrt{|\mathcal{A}||\mathcal{B}|} \geq c(n-1)!\), then there exist \(i,j \in [n]\) such that all but at most \(O((n-2)!)\) permutations in \(\mathcal{A}\) and all but at most \(O((n-2)!)\) permutations in \(\mathcal{B}\) map \(i\) to \(j\).

\begin{theorem}
\label{thm:roughstabilitycross}
Let \(c > 0\) be a positive constant. If \(\mathcal{A},\mathcal{B} \subset S_{n}\) are cross-intersecting with \(\sqrt{|\mathcal{A}||\mathcal{B}|} \geq c(n-1)!\), then there exist \(i,j \in [n]\) such that all but at most \(O((n-2)!)\) permutations in \(\mathcal{A}\) and all but at most \(O((n-2)!)\) permutations in \(\mathcal{B}\) map \(i\) to \(j\).
\end{theorem}
\begin{proof}
Let \(|\mathcal{A}| \leq |\mathcal{B}|\). First, we adapt the proof of Theorem \ref{thm:crosshoffman} to obtain information about the distances \(D:=||P_{U^{\perp}}(v_{X})||\) and \(E:= ||P_{U^{\perp}}(v_{Y})||\). This time, we have
\[\sum_{i>1:\lambda_{i} \neq \lambda_{N}} \xi_{i}^{2}= D^{2};\]
\[\sum_{i>1:\lambda_{i} \neq \lambda_{N}} \eta_{i}^{2}=E^{2};\]
\[\sum_{i>1:\lambda_{i} = \lambda_{N}}\xi_{i}^{2} = \alpha-\alpha^{2}-D^{2};\]
\[\sum_{i>1:\lambda_{i} = \lambda_{N}}\eta_{i}^{2} = \beta-\beta^{2}-E^{2}.\]
Substituting into (\ref{eq:sum}) gives:
\begin{eqnarray*}
d \alpha \beta & = & -\sum_{i>1:\lambda_{i} \neq \lambda_{N}} \lambda_{i}\xi_{i}\eta_{i}-\lambda_{N}\sum_{i>1:\lambda_{i}=\lambda_{N}} \xi_{i}\eta_{i}\\
& \leq & \mu \sum_{i>1:\lambda_{i} \neq \lambda_{N}} |\xi_{i}||\eta_{i}|+|\lambda_{N}|\sum_{i>1:\lambda_{i}=\lambda_{N}} |\xi_{i}||\eta_{i}|\\
& \leq & \mu \sqrt{\sum_{i>1:\lambda_{i}\neq \lambda_{N}} \xi_{i}^{2}}\sqrt{\sum_{i>1:\lambda_{i} \neq \lambda_{N}} \eta_{i}^{2}} + |\lambda_{N}| \sqrt{\sum_{i>1:\lambda_{i}=\lambda_{N}} \xi_{i}^{2}}\sqrt{\sum_{i>1:\lambda_{i}=\lambda_{N}} \eta_{i}^{2}}\\
& = & \mu DE + |\lambda_{N}| \sqrt{\alpha-\alpha^{2}-D^{2}}\sqrt{\beta-\beta^{2}-E^{2}},
\end{eqnarray*}
where \(\mu = \max_{i>1: \lambda_{i} \neq \lambda_{N}} |\lambda_{i}|\). Note that the derangement graph \(\Gamma\) has \(\mu \leq O((n-2)!)\). Hence, applying the above result to a cross-intersecting pair \(\mathcal{A},\mathcal{B} \subset S_{n}\) with \(\sqrt{|\mathcal{A}||\mathcal{B}|} = (1-\delta)(n-1)!\), we obtain
\[\sqrt{1-\alpha-D^{2}/\alpha}\sqrt{1-\beta-E^{2}/\beta} \geq \frac{d_{n} \sqrt{\alpha \beta} - \mu (D/\sqrt{\alpha})(E/\sqrt{\beta})}{|\lambda_{N}|} \geq 1-\delta-O(1/n),\]
and therefore \(1-\alpha-D^{2}/\alpha \geq (1-\delta)^{2} - O(1/n)\), so \(D^{2} \leq \alpha(2\delta - \delta^{2} + O(1/n))\). Replacing \(\delta\) with \(2\delta - \delta^{2} + O(1/n)\) in the proof of Theorem \ref{thm:roughstability}, we see that there exist \(i,j \in [n]\) such that
\[|\mathcal{A}_{i \mapsto j}| \geq \frac{(n-1)!}{4\sqrt{2 (n-1) \log n}} (1-\sqrt{2\delta-\delta^{2}}-O(1/\sqrt{n})) = \omega((n-2)!),\]
since \(\delta < 1-c\). For each \(k \neq j\), the pair \(\mathcal{A}_{i \mapsto j},\mathcal{B}_{i \mapsto k}\) is cross-intersecting, so as in Lemma \ref{lemma:crosslemma}, we have:
\[|\mathcal{A}_{i \mapsto j}||\mathcal{B}_{i \mapsto k}| \leq ((n-2)!)^{2}.\]
Hence, for all \(k \neq j\),
\[|\mathcal{B}_{i \mapsto k}| \leq o((n-2)!),\]
so summing over all \(j \neq k\) gives
\[|\mathcal{B} \setminus \mathcal{B}_{i \mapsto j}| \leq o((n-1)!).\]
Since \(|\mathcal{B}| \geq |\mathcal{A}|\), \(|\mathcal{B}| \geq c(n-1)!\), and therefore
\[|\mathcal{B}_{i \mapsto j}| \geq (c-o(1))(n-1)!.\]
For each \(k \neq j\), the pair
\(\mathcal{A}_{i \mapsto k},\mathcal{B}_{i \mapsto j}\) is cross-intersecting, so as before, we have:
\[|\mathcal{A}_{i \mapsto k}||\mathcal{B}_{i \mapsto j}| \leq ((n-2)!)^{2}.\]
Hence, for all \(k \neq j\),
\[|\mathcal{A}_{i \mapsto k}| \leq O((n-3)!),\]
so summing over all \(j \neq k\) gives
\[|\mathcal{A} \setminus \mathcal{A}_{i \mapsto j}| \leq O((n-2)!).\]
Also, \(|\mathcal{B}| = |\mathcal{B}_{i \mapsto j}| + |\mathcal{B} \setminus \mathcal{B}_{i \mapsto j}| \leq (1+o(1))(n-1)!\), so \(|\mathcal{A}| \geq c^{2}(1-o(1))(n-1)!\). Hence,
\[|\mathcal{A}_{i \mapsto j}| \geq c^{2}(1-o(1))(n-1)!,\]
so by the same argument as above,
\[|\mathcal{B}_{i \mapsto k}| \leq O((n-3)!)\]
for all \(k \neq j\), and therefore
\[|\mathcal{B} \setminus \mathcal{B}_{i \mapsto j}| \leq O((n-2)!)\]
as well, proving Theorem \ref{thm:roughstabilitycross}.
\end{proof}
We may use Theorem \ref{thm:roughstabilitycross} to deduce two Hilton-Milner type results on cross-intersecting families:
\begin{theorem}
\label{thm:stabmincross}
For \(n\) sufficiently large, if \(\mathcal{A},\mathcal{B} \subset S_{n}\) are cross-intersecting but not both contained within the same 1-coset, then
\[\min(|\mathcal{A}|,|\mathcal{B}|) \leq |\mathcal{C}| = (n-1)!-d_{n-1}-d_{n-2}+1,\]
with equality if and only if
\begin{eqnarray*}
\mathcal{A}&=&\{\sigma \in S_{n}: \sigma(i)=j,\ \sigma \textrm{ intersects } \tau\}\cup \{\rho\},\\
\mathcal{B}& =& \{\sigma \in S_{n}: \sigma(i)=j, \sigma \textrm{ intersects } \rho\}\cup \{\tau\}
\end{eqnarray*}
for some \(i,j \in [n]\) and some \(\tau,\rho \in S_{n}\) which intersect and do not map \(i\) to \(j\).
\end{theorem}
\begin{proof}
Let \(\mathcal{A},\mathcal{B} \subset S_{n}\) be cross-intersecting, and not both centred, with \[\min(|\mathcal{A}|,|\mathcal{B}|) \geq |\mathcal{C}|.\]
Applying Theorem \ref{thm:roughstabilitycross} with any \(c < 1-1/e\), we see that there exist \(i,j \in [n]\) such that 
\[|\mathcal{A} \setminus \mathcal{A}_{i \mapsto j}|,|\mathcal{B} \setminus \mathcal{B}_{i \mapsto j}|\leq O((n-2)!).\]
By double translation, we may assume that \(i=j=1\), so
\[|\mathcal{A} \setminus \mathcal{A}_{1 \mapsto 1}|,|\mathcal{B} \setminus \mathcal{B}_{1 \mapsto 1}|\leq O((n-2)!).\]
Assume \(\mathcal{A}\) is not contained within the 1-coset \(\{\sigma \in S_{n}:\ \sigma(1)=1\}\); let \(\rho\) be a permutation in \(\mathcal{A}\) not fixing 1. Suppose for a contradiction that \(\mathcal{A}\) contains another permutation \(\pi\) not fixing 1. As in the proof of Theorem \ref{thm:cameronkuconj}, this implies that
\[|\mathcal{B}_{1 \mapsto 1}| \leq (n-1)!-d_{n-1}-2d_{n-2},\]
and so by assumption,
\[|\mathcal{B}\setminus \mathcal{B}_{1 \mapsto 1}| \geq d_{n-2}+1,\]
so \(\mathcal{B} \setminus \mathcal{B}_{1 \mapsto 1}\) cannot be \((\log n)\)-intersecting. As in the proof of Theorem \ref{thm:cameronkuconj}, this implies that
\[|\mathcal{A}_{1 \mapsto 1}| \leq ((1-1/e)^{2}+o(1))(n-1)!,\]
giving
\[|\mathcal{A}| \leq ((1-1/e)^{2}+o(1))(n-1)! < |\mathcal{C}|\]
---a contradiction. Hence,
\[\mathcal{A} = \mathcal{A}_{1 \mapsto 1} \cup \{\rho\}.\]
If \(\mathcal{B}\) were centred, then every permutation in \(\mathcal{B}\) would have to fix 1 and intersect \(\rho\), and we would have \(|\mathcal{B}| = |\mathcal{B}_{1 \mapsto 1}| \leq (n-1)!-d_{n-1}-d_{n-2} < |\mathcal{C}|\), a contradiction. Hence, \(\mathcal{B}\) is also non-centred. Repeating the above argument with \(\mathcal{B}\) in place of \(\mathcal{A}\), we see that \(\mathcal{B}\) contains just one permutation not fixing 1, \(\tau\) say. Hence,
\[\mathcal{B} = \mathcal{B}_{1 \mapsto 1} \cup \{\tau\}.\]
Since \(\min(|\mathcal{A}|,|\mathcal{B}|) \geq |\mathcal{C}|\), we have
\begin{eqnarray*}
\mathcal{A}_{1 \mapsto 1}&=&\{\sigma \in S_{n}: \sigma(1)=1,\ \sigma \textrm{ intersects } \tau\},\\
\mathcal{B}_{1 \mapsto 1}& =& \{\sigma \in S_{n}: \sigma(1)=1,\ \sigma \textrm{ intersects } \rho \},
\end{eqnarray*}
proving the theorem.
\end{proof}

Similarly, we may prove

\begin{theorem}
For \(n\) sufficiently large, if \(\mathcal{A},\mathcal{B} \subset S_{n}\) are cross-intersecting but not both contained within the same 1-coset, then
\[|\mathcal{A}||\mathcal{B}| \leq ((n-1)!-d_{n-1}-d_{n-2})((n-1)!+1),\]
with equality if and only if
\[\mathcal{A}=\{\sigma \in S_{n}: \sigma(i)=j,\ \sigma \textrm{ intersects } \rho\},\quad \mathcal{B} = \{\sigma \in S_{n}: \sigma(i)=j\}\cup\{\rho\}\]
for some \(i,j \in [n]\) and some \(\rho \in S_{n}\) with \(\rho(i) \neq j\).
\end{theorem}
\begin{proof}
Let \(\mathcal{A},\mathcal{B} \subset S_{n}\) be cross-intersecting, and not both centred, with
\[|\mathcal{A}||\mathcal{B}| \geq ((n-1)!-d_{n-1}-d_{n-2})((n-1)!+1).\]
We have
\[\sqrt{|\mathcal{A}||\mathcal{B}|} \geq (\sqrt{1-1/e}-O(1/n))(n-1)!,\]
so applying Theorem \ref{thm:roughstabilitycross} with any \(c < \sqrt{1-1/e}\), we see that there exist \(i,j \in [n]\) such that
\[|\mathcal{A} \setminus \mathcal{A}_{i \mapsto j}|,|\mathcal{B} \setminus \mathcal{B}_{i \mapsto j}|\leq O((n-2)!).\]
By double translation, we may assume that \(i=j=1\), so
\[|\mathcal{A} \setminus \mathcal{A}_{1 \mapsto 1}|,|\mathcal{B} \setminus \mathcal{B}_{1 \mapsto 1}|\leq O((n-2)!).\]
Therefore,
\begin{equation}
\label{eq:prodbd}
\sqrt{|\mathcal{A}_{1 \mapsto 1}||\mathcal{B}_{1 \mapsto 1}|} \geq (\sqrt{1-1/e}-O(1/n))(n-1)!.
\end{equation}
If \(\mathcal{B}\) contains some permutation \(\rho\) not fixing 1, then
\[\mathcal{A}_{1 \mapsto 1} \subset \{\sigma \in S_{n}: \sigma(1)=1, \sigma \textrm{ intersects }\rho\},\]
and therefore
\[|\mathcal{A}_{1 \mapsto 1}| \leq (n-1)!-d_{n-1}-d_{n-2} = (1-1/e+o(1))(n-1)!.\]
Similarly, if \(\mathcal{A}\) contains a permutation not fixing 1, then
\[|\mathcal{B}_{1 \mapsto 1}| \leq (1-1/e+o(1))(n-1)!.\]
By (\ref{eq:prodbd}), both statements cannot hold (provided \(n\) is large), so we may assume that every permutation in \(\mathcal{A}\) fixes 1, and that \(\mathcal{B}\) contains some permutation \(\rho\) not fixing 1. Hence,
\[\mathcal{A} \subset \{\sigma \in S_{n}: \sigma(1)=1, \sigma \textrm{ intersects }\rho\},\]
and
\begin{equation}
\label{eq:Asmall}
|\mathcal{A}| \leq (n-1)!-d_{n-1}-d_{n-2} = (1-1/e+o(1))(n-1)!.
\end{equation}
So by assumption,
\begin{equation}
\label{eq:Blarge}
|\mathcal{B}| \geq (n-1)!+1.
\end{equation}
Suppose for a contradiction that \(\mathcal{B}\) contains another permutation \(\pi \neq \rho\) such that \(\pi(1) \neq 1\). Then, by the same argument as in the proof of Theorem \ref{thm:cameronkuconj}, we would have
\[|\mathcal{A}|=|\mathcal{A}_{1 \mapsto 1}| \leq (n-1)!-d_{n-1}-2d_{n-2},\]
so by assumption,
\[|\mathcal{B}| \geq \frac{((n-1)!-d_{n-1}-d_{n-2})((n-1)!+1)}{(n-1)!-d_{n-1}-2d_{n-2}} = (n-1)!+\Omega((n-2)!).\]
This implies that \(|\mathcal{B}\setminus \mathcal{B}_{1 \mapsto 1}| = \Omega((n-2)!)\), so \(\mathcal{B} \setminus \mathcal{B}_{1 \mapsto 1}\) cannot be \((\log n)\)-intersecting. Hence, by the same argument as in the proof of Theorem \ref{thm:cameronkuconj},
\[|\mathcal{A}_{1 \mapsto 1}| \leq ((1-1/e)^{2}+o(1))(n-1)!.\]
Therefore,
\[\sqrt{|\mathcal{A}_{1 \mapsto 1}||\mathcal{B}_{1 \mapsto 1}|} \leq (1-1/e+o(1))(n-1)!\]
--- contradicting (\ref{eq:prodbd}). Hence, \(\rho\) is the only permutation in \(\mathcal{B}\) not fixing 1, i.e.
\[\mathcal{B} = \mathcal{B}_{1 \mapsto 1} \cup \{\rho\}.\]
So we must have equality in (\ref{eq:Blarge}), i.e.
\[\mathcal{B}_{1 \mapsto 1} = \{\sigma \in S_{n}:\ \sigma(1)=1\}.\]
But then we must also have equality in (\ref{eq:Asmall}), i.e.
\[\mathcal{A} = \{\sigma \in S_{n}:\ \sigma(1)=1,\ \sigma \textrm{ intersects }\rho\},\]
proving the theorem.
\end{proof}
\section{Conclusion and open problems}
Due to our use of the martingale inequality in Theorem \ref{thm:maurey}, our proof of the Cameron-Ku conjecture requires \(n > 10^4\), so it is obviously impracticable to check the remaining cases using a computer. It would be interesting to find a proof that works for all \(n \geq 6\); we do not rule out the possibility of a purely combinatorial proof, although we have been unable to find one.

We now turn to the question of \(k\)-intersecting families of permutations. In \cite{jointpaper}, it is proved that for \(n\) sufficiently large depending on \(k\), if \(\mathcal{A} \subset S_n\) is \(k\)-intersecting, then \(|\mathcal{A}| \leq (n-k)!\), with equality only if \(\mathcal{A}\) is a `\(k\)-coset', meaning a family of the form
\[\{\sigma \in S_n:\ \sigma(i_1)=j_1,\sigma(i_2)=j_2,\ldots,\sigma(i_k)=j_k\},\]
for some distinct \(i_1,\ldots,i_k \in [n]\) and distinct \(j_1,\ldots,j_k \in [n]\). One of the most natural open problems in the area is to obtain an analogue of the Ahlswede-Khachatrian theorem (see \cite{ahslwedekhachatrian}) for \(k\)-intersecting families in \(S_{n}\), i.e. to determine the maximum-sized \(k\)-intersecting families in \(S_{n}\) for \textit{every} value of \(n\) and \(k\). We make the following conjecture:
\begin{conjecture}
A maximum-sized \(k\)-intersecting family in \(S_{n}\) must be a double translate of one of the families
\[\mathcal{F}_{i} = \{\sigma \in S_{n}:\ \sigma \textrm{ has at least } k+i \textrm{ fixed points in } [k+2i]\}\ (0 \leq i \leq (n-k)/2).\]
\end{conjecture}

This would imply that the maximum size is \((n-k)!\) for \(n > 2k\). We believe that new techniques will be required to prove the above conjecture.

In \cite{jcta}, the author proves an analogue of the Cameron-Ku conjecture for \(k\)-intersecting families of permutations:
\begin{theorem}
\label{thm:hiltonmilnertype}
For \(n\) sufficiently large depending on \(k\), if \(\mathcal{A} \subset S_{n}\) is a \(k\)-intersecting family which is not contained within a \(k\)-coset, then \(\mathcal{A}\) is no larger than the family
\begin{eqnarray*}
\mathcal{D} & = & \{\sigma \in S_{n}:\ \sigma(i)=i\ \forall i \leq k,\ \sigma(j)=j\ \textrm{for some}\ j > k+1\}\\
&& \cup \{(1\ k+1),(2\ k+1),\ldots,(k \ k+1)\},
\end{eqnarray*}
which has size \((1-1/e+o(1))(n-k)!\). Moreover, if \(\mathcal{A}\) has the same size as \(\mathcal{D}\), then it must be a double translate of \(\mathcal{D}\).
\end{theorem}
The methods used are similar to those in this paper, but the representation-theoretic arguments are substantially more involved. It would also be interesting to obtain an analogue of the complete non-trivial \(k\)-intersection theorem of Ahlswede and Khachatrian in \cite{completenontrivial}. We make the following conjecture:
\begin{conjecture}
For any \(n\) and \(k\), if \(\mathcal{A} \subset S_{n}\) is a \(k\)-intersecting family which is not contained within a \(k\)-coset, and has the maximum size subject to these conditions, then it must be a double translate of the family \(\mathcal{D}\) in Theorem \ref{thm:hiltonmilnertype}, or of one of the \(\mathcal{F}_{i}\)'s.
\end{conjecture}
We now turn to the question of improving Theorem \ref{thm:roughstability}. We conjecture that the hypothesis \(|\mathcal{A}| \geq \Omega((n-1)!)\) is unnecessary; in fact, we make the following:
\begin{conjecture}
If \(\mathcal{A} \subset S_n\) is intersecting, then it requires the removal of at most
\[(n-2)!-(n-3)!\]
permutations to make it centred. If \(n \geq 6\), then equality holds only if \(\mathcal{A}\) is a double translate of
\[\{\sigma \in S_n:\ \sigma \textrm{ has at least 2 fixed points in }\{1,2,3\}\}.\]
\end{conjecture}
We make the analogous conjecture for \(k\)-intersecting families:
\begin{conjecture}
For \(n\) sufficiently large depending on \(k\), if \(\mathcal{A} \subset S_{n}\) is \(k\)-intersecting, then there exists a \(k\)-coset containing all but at most
\[k((n-k-1)!-(n-k-2)!)\]
of the permutations in \(\mathcal{A}\). This is sharp only when \(\mathcal{A}\) is a double translate of \(\mathcal{F}_{1}\).
\end{conjecture}
\subsubsection*{Acknowledgements}
The author is indebted to Ehud Friedgut for many helpful discussions.

David Ellis\\
Department of Pure Mathematics and Mathematical Statistics,\\
Centre for Mathematical Sciences,\\
University of Cambridge,\\
Wilberforce Road,\\
Cambridge,\\
CB3 0WB\\
UK\\
\\
{\em E-mail:} \texttt{dce27@cam.ac.uk}
\end{document}